\documentclass{aptpub}

\authornames{Q. LE GALL, B. B\L{}ASZCZYSZYN, E. CALI AND T. EN-NAJJARY}
\shorttitle{Continuum Line-of-Sight Percolation on Poisson-Voronoi Tessellations}

\usepackage{enumitem}
\usepackage{graphicx}
\usepackage{amsmath}
\usepackage{amssymb}
\usepackage{amsfonts}
\usepackage{mathtools} 
\usepackage{mathrsfs} 
\usepackage{color}
\usepackage{stmaryrd}
\usepackage{bbm}
\usepackage{tikz}
\usetikzlibrary{decorations}
\usepackage{pgfplots}
\pgfplotsset{compat=newest}
\pgfplotsset{every axis legend/.append style={
at={(0,0)},
anchor=north east}} 
\usepgfplotslibrary{fillbetween}
\usepackage{url}

\newcommand{\Probab}{\mathbb{P}}

\newcommand{\R}{\mathbb{R}}
\newcommand{\indep}{\rotatebox[origin=c]{90}{$\models$}}

\theoremstyle{empty}
\newtheorem{duplicate}{Reformulation}

\begin{document}

\title{Continuum Line-of-Sight Percolation on Poisson-Voronoi Tessellations}

\authorone[Orange Labs Networks]{Quentin Le Gall}

\authortwo[Inria - \'Ecole Normale Sup\'erieure] {Bart\L{}omiej B\L{}aszczyszyn}

\authorone[Orange Labs Networks]{\'Elie Cali}

\authorone[Orange Labs Networks]{Taoufik En-Najjary}

\addressone{Orange Labs Networks, Modelling and Statistical Analysis, 44 avenue de la R\'epublique, CS 50010, 92326 Ch\^atillon Cedex, France}
\addresstwo{Centre de recherches Inria de Paris, DYOGENE, 2 rue Simone Iff, CS 42112, 75589 Paris Cedex 12, France}

\begin{abstract}
   In this work, we study a new model for continuum line-of-sight percolation in a random environment driven by the  Poisson-Voronoi tessellation in the $d$-dimensional Euclidean space. The edges (one-dimensional facets, or simply 1-facets) of this tessellation are the support of a Cox point process, while the vertices (zero-dimensional facets or simply 0-facets) are the support of a Bernoulli point process. Taking the superposition $Z$ of these two processes, two points of $Z$ are linked by an edge if and only if they are sufficiently close and located on the same edge (1-facet) of the supporting tessellation. We study the percolation of the random graph arising from this construction and prove that a 0-1 law, a subcritical phase as well as a supercritical phase exist under general assumptions. Our proofs are based on a coarse-graining argument  with some notion of stabilization and asymptotic essential connectedness  to investigate continuum percolation for Cox point processes. We also give numerical estimates of the critical parameters of the model in the planar case, where our model is intended to represent telecommunications networks in a random environment with obstructive conditions for signal propagation.  
\end{abstract}



\setcounter{footnote}{0}
\renewcommand\thefootnote{\arabic{footnote}}

\section{Introduction}
\subsection{Background and motivation}
Bernoulli bond percolation, introduced in the late fifties by Broadbent and Hammersley \cite{broadbent1957percolation}, is one of the simplest mathematical models featuring phase transition. Since then, this model has been generalized in many different ways, making percolation theory a broader and still very active research topic today. 
\\ \indent Since the seminal work~\cite{gilbert1961random} of Gilbert, random graphs have been a key mathematical tool for the modelling of telecommunications networks. Good connectivity of the network is then interpreted by percolation of its associated connectivity graph. Over the years, Gilbert's original model has been refined and lots of mathematical models for telecommunications networks are now available in the literature.
\\ \indent In a  recent work~\cite{LeGal1904:Influence},   a new mathematical model for the so-called {\em device-to-device} wireless  networks in obstructive urban environments was proposed and studied numerically.
It is meant to represent direct wireless connections between users
(and possibly some relays) taking place in  an urban environment,
with limited range connections possible only  within line-of-sight along the streets. Such  obstructive conditions for signal propagation are sometimes called {\em urban canyon shadowing}.
In this paper, we study the percolation of this new  model on a more theoretical approach.

More precisely, we model the system of streets as the edges (i.e. one-dimensional facets or simply 1-facets) of the Poisson-Voronoi tessellation (PVT) in the $d$-dimensional Euclidean space.  These edges (which will also be called streets) form the random support of a Cox point process modelling users of the network, with a constant  density~$\lambda$ of users per unit length of  street. Moreover, the vertices of the PVT (i.e. zero-dimensional facets, which will also be called crossroads) are the support of a Bernoulli point process  modelling relays. These relays are assumed to be conditionally independent of the users given the realization of the PVT and
the (conditional) probability of having
a relay on a given PVT vertex is denoted by~$p$.
The superposition of these two point processes (users and relays), denoted by $Z$, defines the nodes of the connectivity graph denoted by~$\mathcal{G}$, with edges existing between any two nodes  
that  are located on the \emph{same} edge of the PVT and  closer to each other than some threshold distance $r$ called connectivity range.

This new percolation model can be seen as a superposition of some more basic models. Indeed, the Bernoulli process of the relays alone, with infinite connectivity range ($r=\infty$), corresponds to the {\em PVT site percolation model} that  has already attracted some attention in the literature, see e.g.~\cite{becker_percolation_2009,neher2008topological}. 
 On the other hand, taking all relays ($p=1$) with  a finite connectivity range
  ($r<\infty$) and no users on streets ($\lambda=0$) corresponds to some 
  PVT bond percolation model with an edge open if its length is smaller than~$r$, equivalently to the Gilbert graph~\cite{gilbert1961random}, considered here on the point process of the vertices of the PVT.  To the best of our knowledge, this model, which we call {\em PVT hard-geometric bond percolation}, has not been considered before.  Finally, adding users (taking $\lambda>0$) introduces  random, conditionally independent, openings of some arbitrarily long streets, where the opening probability depends (in some particular way) on the street length.
  This is similar to the classical random-connection model~\cite[Chapter~6]{meester_continuum_1996}, considered here on the point process of the vertices of the PVT. 
    Again, to the best of our knowledge, this model, which we call {\em PVT soft-geometric bond percolation}, has not been considered before. Both of these models (PVT hard-geometric and soft-geometric bond percolation) can be seen as random connection models in a line-of-sight environment, given the random support $S$.

Our main findings regarding the general connectivity graph~$\mathcal{G}$  can be summarized as follows:
\begin{itemize}
\item \textbf{0-1 law for the percolation probability of the connectivity graph~$\mathcal{G}$:} The probability that the connectivity graph~$\mathcal{G}$ percolates is either 0 or 1. This is a consequence of the fact that the superposition $Z$ of the point processes of users and relays is mixing, and hence ergodic.
    \item \textbf{Critical probability of  the Bernoulli relay process:} There exists a minimal value of  the parameter of the Bernoulli process $p^*\in(0,1)$ under which percolation of the connectivity graph~$\mathcal{G}$ cannot happen with positive probability, regardless of all other parameters. This is a consequence of the non-triviality of the PVT site percolation threshold. Although, it has  been estimated numerically many times in the literature (see e.g.~\cite{becker_percolation_2009,neher2008topological}), we were not able to find any proof of this fact in the literature. Note that the PVT site percolation model should not be confused with  the {\em Voronoi tiling percolation} model, which 
    consists in coloring each \emph{cell} of a PVT in black independently from all other cells with some fixed probability $p$ and investigating the random tiling of black cells.
    The percolation of this latter model is studied in~\cite{balister2005percolation} and the
     critical probability 
in the planar case  has been proven to be 1/2 in~\cite{bollobas2006critical}.
\item \textbf{Critical connectivity range:}
   For $p>p^*$, there exists a  critical connectivity range  $r^*=r^*(p)$, separating the following two 
    connectivity regimes:
    \begin{itemize}
     \item {\em permanently supercritical~$\mathcal{G}$:} For $r>r^*(p)$ the graph  $\mathcal{G}$ percolates with positive probability for all $\lambda\ge0$.
    \item {\em user critical~$\mathcal{G}$:} For $r<r^*(p)$ the graph $\mathcal{G}$ exhibits a non-trivial phase transition in $\lambda$, i.e. it  does not percolate for small $\lambda>0$ and percolates with positive probability for large enough  $\lambda<\infty$.
    \end{itemize}
    We prove that the  critical range $r^*(p)$,
    numerically  estimated in~\cite{LeGal1904:Influence}, is non trivial in the sense that $0<r^*(p)<\infty$ for $p>p^*$ large enough including some $p<1$.  
\item As a corollary  we obtain  the existence of a \textbf{non-trivial phase transition in the PVT hard-geometric bond percolation model}.
\end{itemize}

 \vskip \baselineskip \emph{The rest of this paper is organized as follows:} We begin with recalling some related works in Subsection~\ref{Ss.Relatedworks}. Then, we present in details our network model and introduce convenient notations in Section~\ref{S.Model}. In Subsection~\ref{S.Results}, we state our theoretical results in more detail. In Subsection~\ref{S.NumericalSimulations}, we present results of numerical simulations of our model in the planar case, so as to illustrate our main mathematical results. Then, we proceed with the proofs of our results in Section~\ref{s.Proofs}. Finally, we conclude and give perspectives for future work in Section~\ref{S.Conclusion}.
\subsection{Related works}
\label{Ss.Relatedworks}
In~\cite{gilbert1961random}, Gilbert introduced percolation in a continuum setting by considering a planar homogeneous Poisson point process where two points are joined by an edge if and only if they are separated by a distance gap less than a given threshold. 
This model has at the time been considered to be a good candidate for representing a telecommunications network, with the range of the stations being taken into account as a parameter. The Poisson case has now extensively been studied \cite{meester_continuum_1996} and Gilbert's model has recently been extended to other types of point processes, among which sub-Poisson \cite{blaszczyszyn2010connectivity,blaszczyszyn2013clustering,blaszczyszyn2015clustering}, Ginibre \cite{ghosh2016continuum} and Gibbsian \cite{jansen2016continuum,stucki2013continuum}. 
\\ \indent The study of percolation processes living in random environments has only been considered recently and outlined that many standard techniques from Bernoulli or continuum percolation cannot be applied in such cases. 
As a matter of fact, new tools and techniques had to be introduced. In this regard, the papers from Balister and Bollob\'as~\cite{balister2005percolation} and
Bollob\'as and Riordan \cite{bollobas2006critical} on the threshold of Voronoi tiling percolation in the plane are pioneering. Later on,~\cite{ahlberg2016quenched,tassion2016crossing} brought additional results concerning this model. Other percolation models~\cite{vahidi1990first}, tessellations~\cite{bollobas2008percolation} and other random graphs \cite{balister2008percolationknearest, balister2008percolation, beringer2017percolation} have also been considered. A more general study of Bernoulli and first-passage percolation on random tessellations has been conducted in \cite{ziesche2016bernoulli,ziesche2016first}. 
\\ \indent A natural extension of Gilbert's model in a random environment setting is obtained by considering a Cox point process, i.e. a Poisson point process with a random intensity measure. Percolation of Gilbert's model in such a setting has theoretically been studied for the first time in \cite{hirsch2018continuum}. 
\\ \indent In Gilbert's original model, connectivity between two network nodes only depends on their mutual Euclidean distance. This assumption has proven to be quite simplistic for the modelling of real telecommunications networks, where physical phenomena such as interference, fading or shadowing are at stake, making the occurrence of connectivity between two nodes depend on other factors. As a matter of fact, other extensions of Gilbert's work have been considered for a more accurate modelling of telecommunications networks. In particular, percolation of the signal-to-interference-plus-noise ratio (SINR) model on the plane has theoretically been studied in \cite{dousse2006percolation}. In the SINR model, a connection between a pair of points does not only depend on their relative distance anymore but also on the positions of all other nodes of the network. SINR percolation for Cox point processes has only been explored very recently \cite{tobias2018signal}. On a more applied perspective, random tessellations have turned out to yield good fits of real street systems, as has been proven in \cite{gloaguen2006fitting}. Percolation thresholds of the Gilbert graph of Cox processes supported by random tessellations have numerically been investigated in \cite{cali2018percolation}, yielding other interesting applications for telecommunications networks. 
\\ \indent Recently, mathematical models of so-called \emph{line-of-sight} (LOS) networks have been introduced, modelling telecommunications networks in environments with regular obstructions, such as large urban environments or indoor environments. Nodes of the network are then connected when they are sufficiently close 
and when they have line-of-sight access to one another, in other words if no physical obstacle stands between them. In \cite{frieze2009line}, asymptotically tight results on $k$-connectivity of the connectivity graphs arising from such models are studied. Bollob\'as, Janson and Riordan \cite{bollobas2009line} extended these results by introducing a line-of-sight site percolation model on the discrete square lattice $\mathbb{Z}^{2}$ and the two-dimensional $n$-torus $\left[n\right]\times \left[n\right] := \llbracket 1,n \rrbracket \times \llbracket 1,n \rrbracket$ and asymptotical results for the critical probability were derived as well. Interesting connections to Gilbert's continuum percolation model were also investigated. However, the study of line-of-sight percolation in a continuum setting with a random environment has not, 
as far as we know, been studied yet. \\ \indent It is in light of these recent developments that we introduced in our previous work~\cite{LeGal1904:Influence} a new percolation model (originally in the planar case) for Cox processes supported by Poisson-Voronoi tessellations (PVT).

\section{Model  and main results}
\label{S.Model}
\subsection{Connectivity graph}
\label{Ss.NetworkModel}

Let $d\geq 2$, $\lambda_{S} > 0$ and $X_{S}$ be a homogeneous Poisson point process (PPP) in the state space $\mathbb{R}^{d}$ with intensity $\lambda_S \in (0,\infty)$. Consider the Poisson-Voronoi tessellation (PVT) $S$ associated with $X_S$. In particular, $S$ is stationary and isotropic. By analogy with a telecommunications network, $S$ will be called \emph{street system} from now onwards. 

Denote by $E \coloneqq (e_{i})_{i \geq 1}$ the edge-set of $S$ and by $V \coloneqq (v_{i})_{i \geq 1}$ the vertex-set of $S$. In other words, the elements of $E$ are the \emph{1-facets} of $S$ and the elements of $V$ are the \emph{0-facets} of $S$.  Furthering the analogy with a telecommunications network, the elements of $E$ (respectively $V$) will be called \emph{streets} (respectively \emph{crossroads}). 

Let $\Lambda(dx) := \nu_{1}(S \cap dx)$, where $\nu_{1}$ denotes the 1-dimensional Hausdorff measure of $\mathbb{R}^{d}$.
Observe that $\Lambda$ is a stationary random measure  on  $\mathbb{R}^d$ ($\Lambda(B)$ is the total edge length of $S$ contained in any Borel set $B \subset \mathbb{R}^d$)
with finite non-null intensity $\gamma:=\mathbb{E}[\Lambda[0,1]^d] \in (0,\infty)$. Note that $\gamma$ also be interpreted as the total edge length per unit volume: for this reason, we call $\gamma$ the \emph{street intensity}. In the planar case $(d=2)$, it is known that $\gamma =  2\sqrt{\lambda_S}$ (see \cite[Section 9.7.2]{chiu_stochastic_2013}). More general formulae for $\gamma$ in any dimension $d \geq 2$ can be found in~\cite{moller_random_1989, moller_lectures_2012, muche2005poisson}.

For $\lambda>0$ consider a  Cox point process $X^{\lambda}$ driven by the random intensity measure $\lambda \Lambda$. 
In other words, conditioned on a given realization of the street system~$S$, $X^{\lambda}$ is a PPP with mean measure $\lambda \Lambda$. 
In accordance with the telecommunication interpretation of the PVT as a street system, the points of $X^{\lambda}$ model locations of users (equipped with mobile devices). In particular, the number of users on a given street  $e \in E$ is a Poisson random variable with mean $\lambda\nu_{1}(e)$ and the numbers of users on two disjoint subsets of $E$ are independent random variables. \\
\indent We consider yet another point process on the PVT, namely  a doubly stochastic Bernoulli point process $Y$ on the set of crossroads $V$ with parameter $p$: 
conditioned on $\Lambda$ (or, equivalently, the PVT $S$) points of $Y$ are placed on the crossroads $V$ of $S$ independently   with probability $p$.
The points of $Y$ will be called (fixed) relays. 
We also assume that the processes of users and of relays  are conditionally independent given their random support, i.e. $X^{\lambda} \indep Y \, \vert \,  \Lambda$. We denote $Z \coloneqq X^{\lambda} \cup Y$ the superposition of users and relays.

We consider an undirected {\em connectivity graph}  
$\mathcal{G}$ 
with the set of vertices given by the points of the point process  $Z$ and the  edges $Z_{i} \leftrightsquigarrow Z_{j}, i \neq j$, if and only if $Z_{i}$ and $Z_{j}$ are located on the \emph{same} street and of mutual Euclidean distance less than $r$:
\begin{equation}
\label{connectivity_mechanism}
\forall \, i \neq j, \: Z_{i} \leftrightsquigarrow Z_{j} \Leftrightarrow 
\left\{
\begin{array}{l}
\exists \, e \in E, \, Z_{i} \in e  \  \text{and} \  Z_{j} \in e \\
\lVert Z_{i} - Z_{j} \rVert \leq r.
\end{array}
\right. 
\end{equation}
Figure~\ref{fig:Network} illustrates a realisation of our network model and of its connectivity graph in the planar case ($d=2$).
\begin{figure}[t]
\centering
    \includegraphics[width=0.5\linewidth]{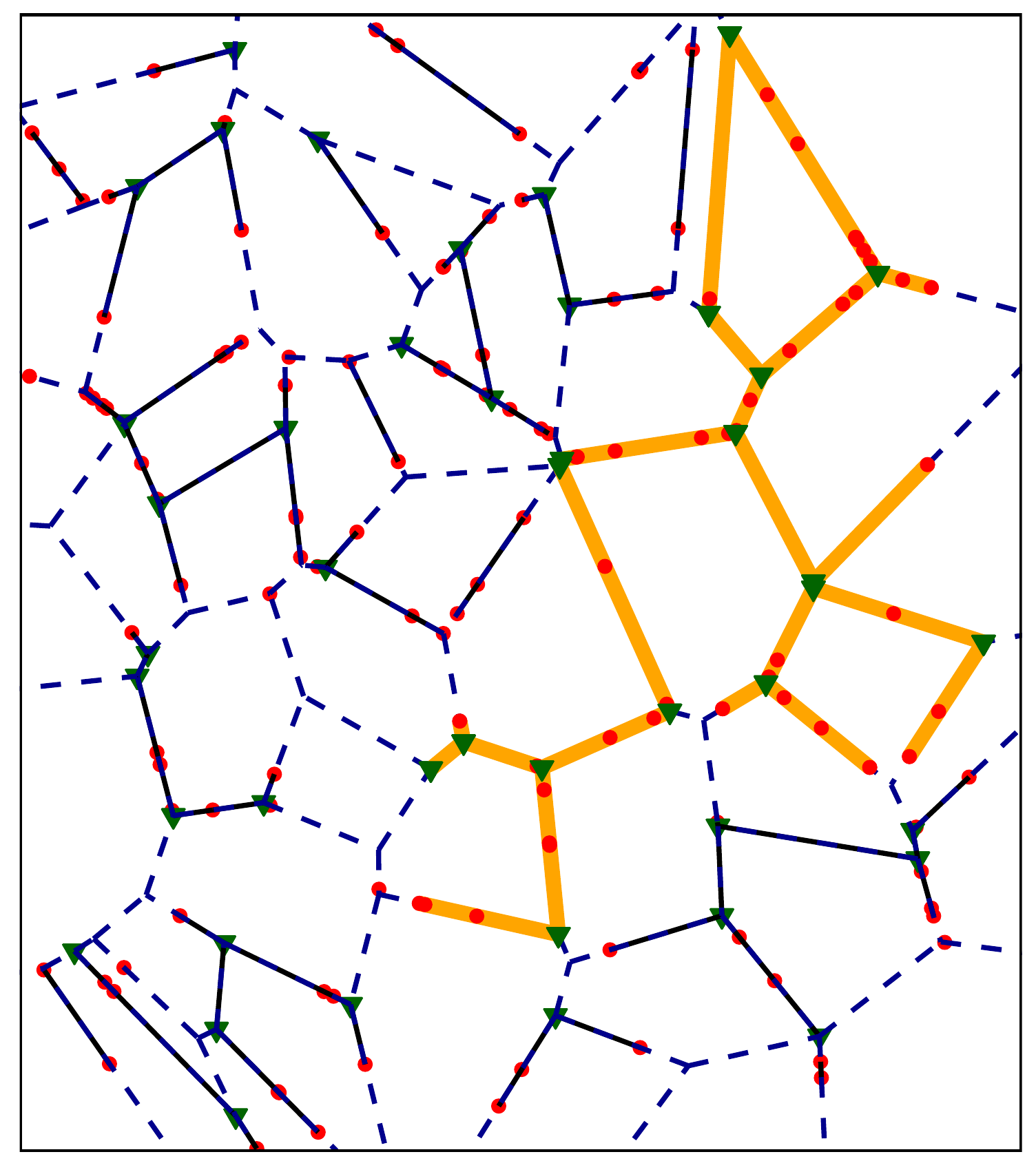}
    \caption{Illustration of the connectivity graph~$\mathcal{G}$ in a bounded simulation window of the plane $\mathbb{R}^2$. The dashed blue lines illustrate the PVT supporting the Cox point process of users represented by red circle-shaped points. The Bernoulli point process of relays is illustrated by the green triangle-shaped points. Finally, edges of the largest connected component of $\mathcal{G}$ in the simulation window are the thicker solid segments highlighted in orange, while edges of smaller connected components are the thinner solid segments highlighted in black. These connected components illustrate the connectivity mechanism given by~\eqref{connectivity_mechanism}.}
    \label{fig:Network}
\end{figure}

In this paper, we study the  percolation properties of the graph $\mathcal{G}$. More precisely,  we identify regimes (i.e. sets of  model parameters $p,\lambda,r$)   where  $\mathcal{G}$ percolates (i.e. has an infinite component) with positive probability.

\subsection{Dimensionless  model parameters}
\label{s.dimensionless-parameters}
While our original percolation model parameters are $\gamma, p,\lambda$ and $r$, it is customary to introduce the following dimensionless model parameters by proceeding as follows. Such parameters turn out to be way more convenient for numerical simulations, as they are also scale-invariant.  \\
\indent Denote by $\bar{l}$ the \emph{mean length of the typical street} (that is to say, the mean length of the typical point of the process of 1-facets of $S$, see~\cite{okabe_spatial_1992}). A general formula for $\bar{l}$ in any dimension $d\geq 2$ is available in~\cite{moller_random_1989, muche2005poisson}. In particular, there exists a positive constant $\kappa:=\kappa(d)$ such that $\bar{l}=\kappa(d)\lambda_S^{-1/d}$. Now, introduce the following dimensionless parameters:
\begin{equation*}
U: = \lambda \bar{l} \quad \text{and} \quad H:=\frac{\bar{l}}{r}.  
\end{equation*}
$U$ corresponds to the mean number of users per typical street, while $H$ is the mean number of hops (of length $r$) required to traverse the typical street.
It is easily shown that for all $d \geq 2$, $\bar{l}=\infty \Leftrightarrow \gamma = 0$ and $\bar{l}=0 \Leftrightarrow \gamma = \infty$. In what follows, we will thus denote by 
$\mathcal{G}_{p,U,H}$ the connectivity graph $\mathcal{G}$
as a function of the model parameters~$(p,U,H)$
in the domain $p \in \left[0;1\right],U\ge 0,H\ge 0$, with $H=0$ interpreted as
$r=\infty$ for some $0<\gamma<\infty$ and 
$U=0$ interpreted as
$\lambda=0$ for some $0<\gamma<\infty$.
Note that since $S$ is stationary and isotropic, percolation of the connectivity graph $\mathcal{G}$ does actually not depend on the parameter $\gamma$. If needed, we can thus fix a value $\gamma$ once and for all (e.g. $\gamma=1$), so that there is a one-to-one map $(p,\lambda,r) \mapsto (p,U,H)$ between the original and the dimensionless parameters.

\subsection{Results}
\label{S.Results}
Denote 
$$P(p,U,H):=\Probab(\mathcal{G}_{p,U,H}\ \text{percolates})$$
and observe $P$ is increasing in~$p$ and $U$ and decreasing in~$H$.

Our first result is an ergodicity result. More precisely, we have the following:

\begin{proposition}
\label{thm-ergodicity}
The superposition $Z$ of the point processes of users and relays is mixing, and hence ergodic.
\end{proposition}

Since the percolation of the connectivity graph $\mathcal{G}$ is a translation-invariant event, a straightforward consequence of the previous result is the following 0-1 law:

\begin{corollary}
\begin{equation*}
    \forall \gamma >0, p \in \left[0,1\right], U \geq 0, H \geq 0,\quad P(p,U,H) \in \lbrace 0,1 \rbrace.
\end{equation*}
In other words, percolation of the connectivity graph $\mathcal{G}$ is either  an almost sure or almost impossible event.
\end{corollary}

For given $p\ge0,H\ge0$ consider the following critical value of the mean number of users per street $U$  
\begin{align*}
U_c(p,H):=\inf\{U\ge 0: P(p,U,H)>0\}
\end{align*}
with $U_c(p,H):=\infty$ if $P(p,U,H)=0$ for all $U\ge 0$.

We aim at showing that there is a region  (i.e. a connected subset) of parameters $(p,H)$ such that $0<U_c(p,H)<\infty$. This is the region where the percolation of $\mathcal{G}$ exhibits non-trivial phase transition in the density of users. Existence of this region follows from the following two results.
\begin{theorem}[Existence of sub-critical intensities of users]
\label{Thm.analogue:subcritical}
For large enough \\ $H\in[0,\infty)$ and  small enough $U>0$ (with the thresholds for  $H$ and $U$ not depending on one another) we have  $P(1,U,H)=0$ and, consequently,  $P(p,U,H)=0$ for any $p\in[0,1]$.\end{theorem}

\begin{theorem}[Existence of super-critical intensities of  users]
\label{Thm.analogue:supercritical}
For any $H\in[0,\infty)$, for 
large enough $p\in(0,1)$ and  $U<\infty$
(the thresholds for $p$ and $U$ depend on $H$ but not on one another)
we have $P(p,U,H)>0$. \end{theorem}
The proofs are presented in Section~\ref{s.Proofs}.

\begin{remark} There are three and may be up to five different ranges of parameters $(p,H)$ of interest in our model. \\
The range of parameters $(p,H)$ where $0<U_c(p,H)<\infty$ (upper-right range schematically presented in orange on Figure~\ref{fig.Phase-transiton-diagram}) can be seen as the {\em critical range of $(p,H)$} in the sense that it separates the following two ranges of $(p,H)$: the
{\em (permanently) sub-critical range} (lower range schematically presented in blue on Figure~\ref{fig.Phase-transiton-diagram}), where $\mathcal{G}$ does not percolate whatever large the density of users ($U_c(p,H)=\infty$) and the {\em (permanently) super-critical range} (upper-left range schematically presented in red on Figure~\ref{fig.Phase-transiton-diagram}), where $\mathcal{G}$
percolates with positive probability, whatever small the density of users ($U_c(p,H)=0$). 
We cannot exclude that this latter range contains a non-empty subset of $(p,H)$ such that  $\mathcal{G}$ does not percolate without users ($U=0$) but percolates with positive probability for arbitrarily small density of users,
as depicted on Figure~\ref{fig.Phase-transiton-diagram-bis}.
Moreover, we do not know whether the permanently sub-critical range contains some $p>p^*$, as also depicted on Figure~\ref{fig.Phase-transiton-diagram-bis}. 
Note that we do not know the exact shapes of the curves separating these ranges except that they are monotonic. Even continuity is not known. 
\end{remark}

\begin{figure}[t]
\begin{center}
\begin{tikzpicture}[scale=1]
\begin{axis}[xmin=-0.001,xmax=1.15, ymin=0.45, ymax=1.001, samples=200, 
grid=major,
xtick={0,0.2,0.4,0.6,0.8,1.15},
xticklabels={0,0.2,0.4,0.6,0.8,$\infty$},
xlabel=$H$,
ylabel=$p$,
domain=0:1.15,
extra x ticks={1},
extra x tick style={ tick label style={xshift=0cm,yshift=.28cm, font=\scriptsize\boldmath}},
extra x tick label={{\color{black}{$\mathbin{\!/\mkern-3mu/\!}$}}},
extra y ticks={0.45},
extra y tick style={ tick label style={xshift=0.3cm,yshift=.21cm, rotate=315, font=\scriptsize\boldmath}},
extra y tick label={{\color{black}{$\mathbin{\!/\mkern-3mu/\!}$}}},
legend pos=south east,
style={font=\tiny}]
\addplot[red, thick, name path=pc] {0.71+(1.2-0.71)/(1+10^(4.5*(0.7-x)))}; 
 \addplot[color=red,mark=*, only marks] table
[x expr=\thisrow{H}*0.7/(656.58-136.51)-0.7/(656.58-136.51)*136.51+0, y expr=\thisrow{p}*(1-0.71)/(683.40-104.90)-(1-0.71)/(683.40-104.90)*104.90+0.71]
{pcH.dat};
 \addplot[blue, name path =ptwostar, thick, forget plot] (x, 0.71);
\addplot[color=black,mark=square, only marks] coordinates{(0.743,1)};
\addplot[color=blue,mark=*, only marks, forget plot] coordinates{(1.15,1)};
\addplot[color=blue,mark=o, only marks, forget plot] coordinates{(1.15,0.71)};
\addplot[color=black,mark=diamond, only marks] coordinates{(0,0.71)};
\node[color=black,opacity=1,text opacity=1] at (axis cs:0.3,0.85) {\fcolorbox{white}{white}{$U_c(p,H)=0$}};
\node[color=black,opacity=1,text opacity=1]  at (axis cs:0.89,0.87) {\fcolorbox{white}{white}
{$0<U_c(p,H)<\infty$}};
\node[color=black] at (axis cs:0.4,0.60) {\fcolorbox{white}{white}{$U_c(p,H)=\infty$}};
 \addlegendentry{$p_c(H)$}
\addlegendentry{$p_c(H)$ simulated}
 \addlegendentry{$H_c=H_0\approx0.743$}
 \addlegendentry{$p^*=p_c(0)\approx0.713$}
\addplot[orange!20, opacity=0.5,forget plot] fill between[of=pc and ptwostar];
\addplot[draw=none,forget plot, name path=xaxis] (x,0.45); 
\addplot[blue!20,opacity=0.5,forget plot] fill between[of=ptwostar and xaxis];
\addplot[draw=none,forget plot, name path=xaxisup] (x,1); 
\addplot[red!20,opacity=0.5,forget plot] fill between[of=xaxisup and pc];
  \end{axis}
\end{tikzpicture}
 \end{center}
\caption{\label{fig.Phase-transiton-diagram}
Conjectured phase transition diagram of $\mathcal{G}$. The existence of the range $0<U_c(p,H)<\infty$ follows from Theorems~\ref{Thm.analogue:subcritical} and~\ref{Thm.analogue:supercritical}. Estimated values for $H \mapsto p_c(H)$, $H_c$ and $p^*$ are obtained in the planar case via Monte-Carlo simulations, see Subsection~\ref{S.NumericalSimulations} and the earlier publication~\cite{LeGal1904:Influence}.}
\end{figure}
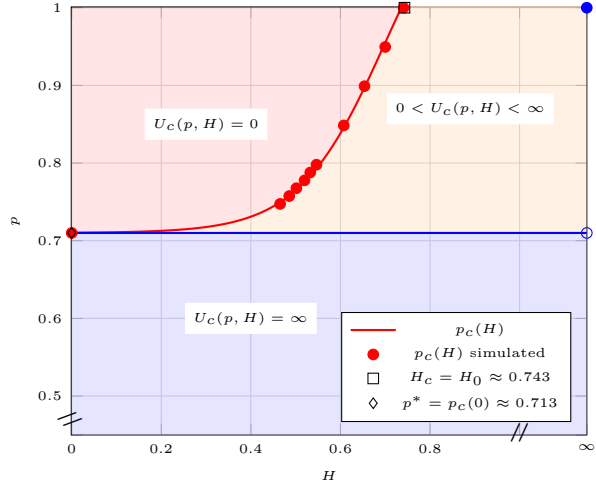
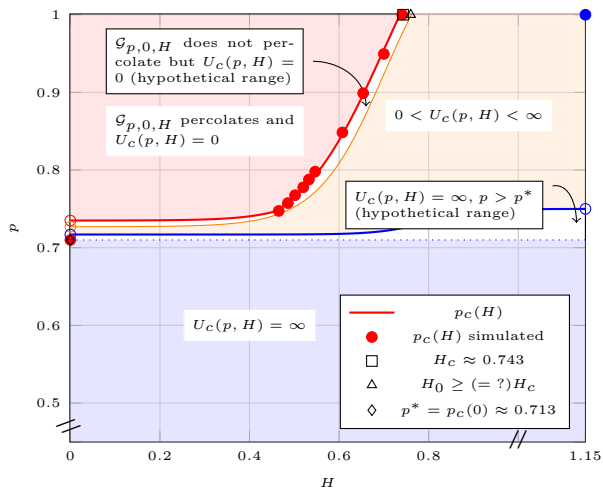
\begin{figure}[t]
\centerline{
\begin{tikzpicture}[scale=1]
\begin{axis}[xmin=-0.001,xmax=1.15, ymin=0.45, ymax=1.001, samples=200, 
grid=major,
xtick={0,0.2,0.4,0.6,0.8,1.15},
xlabel=$H$,
ylabel=$p$,
domain=0:1.15,
extra x ticks={1},
extra x tick style={ tick label style={xshift=0cm,yshift=.28cm, font=\scriptsize\boldmath}},
extra x tick label={{\color{black}{$\mathbin{\!/\mkern-3mu/\!}$}}},
extra y ticks={0.45},
extra y tick style={ tick label style={xshift=0.3cm,yshift=.21cm, rotate=315, font=\scriptsize\boldmath}},
extra y tick label={{\color{black}{$\mathbin{\!/\mkern-3mu/\!}$}}},
legend pos=south east,
style={font=\tiny}]
\addplot[red,  thick, name path=pc] {x<0.452?
 0.735+(1.26-0.735)/(1+10^(6*(0.73-x))):
0.71+(1.2-0.71)/(1+10^(4.5*(0.7-x)))}; 
\addplot[color=red,mark=*, only marks] table [x expr=\thisrow{H}*0.7/(656.58-136.51)-0.7/(656.58-136.51)*136.51+0, y expr=\thisrow{p}*(1-0.71)/(683.40-104.90)-(1-0.71)/(683.40-104.90)*104.90+0.71]{pcH.dat};
\addplot[color=red,mark=o, only marks, forget plot] coordinates{(0,0.735)};
\addplot[orange, name path =pstar, forget plot] {0.727+(1.12-0.727)/(1+10^(5.8*(0.7-x)))};
\addplot[color=orange,mark=o, only marks, forget plot] coordinates{(0,0.727)};
 \addplot[blue, name path =ptwostar, thick, forget plot] {0.717+(0.75-0.717)/(1+10^(8*(0.8-x)))};
 \addplot[color=blue,mark=o, only marks, forget plot] coordinates{(0,0.717)};
 \addplot[color=blue,mark=o, only marks, forget plot] coordinates{(1.15,0.75)};
 \addplot[color=blue,mark=*, only marks, forget plot] coordinates{(1.15,1)};
 \addplot[blue, dotted, name path =pthrestar, forget plot] (x,0.71);
\addplot[color=black,mark=square, only marks] coordinates{(0.743,1)};
 \addplot[color=black,mark=triangle, only marks] coordinates{(0.76,1)};
\addplot[color=black,mark=diamond, only marks] coordinates{(0,0.71)};
\node[color=black,opacity=1,text opacity=1] at (axis cs:0.3,0.85) {\fcolorbox{white}{white}{\parbox{10em}{$\mathcal{G}_{p,0,H}$ percolates and $U_c(p,H)=0$}}};
\node[color=black,opacity=1,text opacity=1] (source) at (axis cs:0.89,0.87) {\fcolorbox{white}{white}
{$0<U_c(p,H)<\infty$}};
\node[color=black,opacity=1,text opacity=1] (source) at (axis cs:0.835,0.755) {\fcolorbox{black}{white}
{\parbox{10em}{$U_c(p,H)=\infty$, $p>p^*$
(hypothetical range)}}};
\node (destination) at (axis cs:1.13,0.72){};
\draw[color=black,->](source) to [out=10,in=90] (destination);
\node[color=black, opacity=1,text opacity=1] (source) at (axis cs:0.3,0.94) {\fcolorbox{black}{white}{\parbox{10em}{$\mathcal{G}_{p,0,H}$ does not percolate but $U_c(p,H)=0$
(hypothetical range)}}};
\node (destination) at (axis cs:0.66,0.87){};
\draw[color=black,->](source) to [out=0,in=90] (destination);
\node[color=black] at (axis cs:0.4,0.60) {\fcolorbox{white}{white}{$U_c(p,H)=\infty$}};
 \addlegendentry{$p_c(H)$}
\addlegendentry{$p_c(H)$ simulated}
 \addlegendentry{$H_c\approx0.743$}
  \addlegendentry{$H_0\ge (=\text{?})H_c$}
 \addlegendentry{$p^*=p_c(0)\approx0.713$}
\addplot[orange!20, opacity=0.5,forget plot] fill between[of=pstar and ptwostar];
\addplot[draw=none,forget plot, name path=xaxis] (x,0.45); 
\addplot[blue!20,opacity=0.5,forget plot] fill between[of=pthrestar and xaxis];
\addplot[draw=none,forget plot, name path=xaxisup] (x,1); 
\addplot[red!20,opacity=0.5,forget plot] fill between[of=xaxisup and pc];
  \end{axis}
\end{tikzpicture}}
\caption{\label{fig.Phase-transiton-diagram-bis}
Phase transition diagram of $\mathcal{G}$ with hypothetical ranges of $(p,H)$. Estimated values for $H \mapsto p_c(H)$, $H_c$ and $p^*$ are obtained in the planar case via Monte-Carlo simulations, see Subsection~\ref{S.NumericalSimulations} and the earlier publication~\cite{LeGal1904:Influence}.}
\end{figure}

In what follows we discuss  some special cases of our percolation model.

\subsubsection{PVT site percolation}
\label{sss.PVT-site-percolation}
Note that for $H=0$ ($r=\infty$ for some $\gamma>0$),
$P(p,U,0)=:P_{PVT}(p)$ does not depend on $U$
and corresponds to the probability of the 
(i.i.d.) site percolation model on the (dimensionless) planar PVT.
Denote the critical parameter of this model by
\begin{equation*}
p^*:=\inf \lbrace {p}\in[0,1]:
P_{PVT}(p)>0 \rbrace.
\end{equation*}
\indent Clearly, by the monotonicity of the model $\mathcal{G}_{p,U,H}$ does not percolate for  $p<p^*$, whatever $U\ge0$, $H\ge0$. Moreover, as a consequence of Theorem~\ref{Thm.analogue:supercritical} and of standard percolation arguments, we obtain the non-triviality of the PVT site percolation threshold:

\begin{proposition}
\label{prop-non-triviality-PVT-site-threshold}
$p^* \in (0,1).$
\end{proposition}

\subsubsection{PVT Hard-geometric bond percolation} 
For $U=0$ (no mobile users) 
$\mathcal{G}_{1,0,H}$ corresponds  to a  (non-standard) inhomogeneous bond percolation model on the PVT, in which the  edges of the PVT are open or closed depending whether their length is smaller or larger than some threshold.  We call it PVT hard-geometric bond percolation.
It seems that this model has not been studied in the 
literature.

Define the critical bond parameter of this model
\begin{equation*}
H_c \coloneqq \sup \lbrace {H} \geq 0: 
P(1,0,H)>0
\rbrace.
\end{equation*}
Note that $H_c \le H_0 \coloneqq \sup \lbrace H \geq 0: U_c(1,H)=0
\rbrace$ and, by Theorem~\ref{Thm.analogue:subcritical}, $H_0<\infty$. 
This observation, combined with the following result, ensures that there is a non-trivial phase transition in the PVT hard-geometric bond percolation model, as stated in the introduction:
\begin{theorem}[Existence of the  permanently super-critical range] \label{Thm.permanently-super-critical}
For large enough $p<1$ and small enough $H>0$ (with the thresholds for  $p$ and $H$ not depending on one another), we have that $P(p,0,H)>0$. \end{theorem}

As a consequence of Theorem~\ref{Thm.permanently-super-critical} and the monotonicity of $P(p,0,H)$ with $p$, we immediately have the following result.

\begin{corollary}[Existence of a super-critical phase in the PVT hard-geometric bond percolation]
\label{Coroll.phase-transition-hard-geometric}
For small enough $H>0$, we have that $P(1,0,H)>0$. Therefore, $H_c >0$.
\end{corollary}

\subsubsection{PVT soft-geometric bond percolation}
Considering  $U>0$ introduces to our model the possibility of opening some 
long edges, which are not open in the PVT hard-geometric bond percolation.
Note that this is equivalent to yet another bond percolation, 
in which the edges of the PVT are open  
independently with probabilities depending  on their lengths. We call it 
soft-geometric bond percolation. It seems that such a model has not been studied in the literature either.

\subsubsection{$\mathcal{G}$ as a superposition of three percolation models}
Note  that $\mathcal{G}_{p,U,H}$ is a superposition of  the three independent (given the PVT) percolation models:
site mode, hard-geometric bond model and soft-geometric bond model.
The reason we cannot exclude the hypothetical ranges of $(p,H)$ 
such that $U_c(p,H)=\infty$ for  $p>p^*$
(see Figure~\ref{fig.Phase-transiton-diagram-bis}) is that we do not know whether
for all $p>p^*$ the percolation of  $\mathcal{G}$ can be preserved 
when lowering the  distance threshold of the hard-geometric bond percolation
(increasing  $H$)  by increasing the probabilities of edge opening  (increasing $U$) in the soft-geometric percolation model. 
This is possible only for large enough $p$ (see Theorem~\ref{Thm.analogue:supercritical}).

\subsection{Numerical observations in the planar case}
\label{S.NumericalSimulations}
\begin{figure}[t!]
\centering
\includegraphics[width=0.5\textwidth]{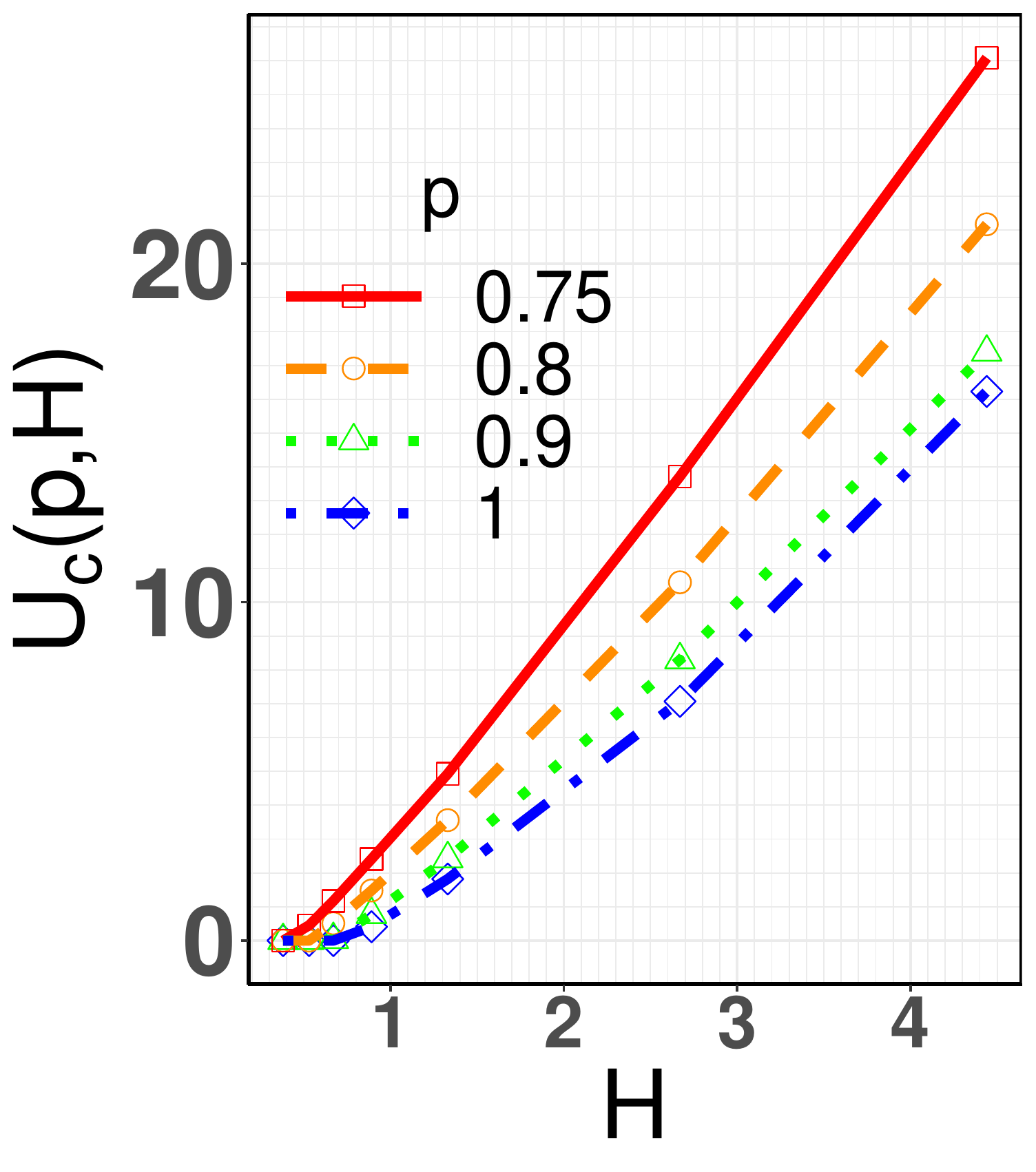}
\caption{\label{fig.Uc} Numerical estimations of $U_c(p,H)$ for some values of $(p,H)$.}
\end{figure}

In what follows, for the sake of visualisation, we briefly recall some numerical findings in the planar case (i.e. setting $d=2$) obtained  in an earlier publication~\cite{LeGal1904:Influence},
to which we refer the reader for further information regarding the simulation methodology. Note that in the planar case, it is known (see e.g.~\cite{okabe_spatial_1992}) that we have
\begin{equation*}
    \gamma = 2\sqrt{\lambda_S} \qquad ; \qquad \bar{l}=\frac{2}{3\sqrt{\lambda_S}}
\end{equation*}
and, as a consequence of the two previous formulae, we have $\kappa(2)=\frac{2}{3}$ and the length of the typical edge writes $\bar{l}=4/(3\gamma)$. \\

The estimates for the critical values $p^*$ and $H_c$  were found close to    $p^*\approx 0.713$ and $H_c\approx 0.743$. Recall that $H_c$ concerns a model which, as far as we know, has not been studied yet in the literature. Considering the estimate for $p^*$, our value only slightly differs from the most recent estimation available in the literature~\cite{becker_percolation_2009}: $p^* \approx 0.71410 \pm 0.00002$. While the authors providing this estimate proceeded with Monte-Carlo simulations with periodic boundary conditions and investigated the growth of the largest cluster, we chose a crossing-window method to obtain our estimate of $p^*$. 

For $H<H_c$ define 
\begin{equation*}
    p_c(H) \coloneqq \inf \lbrace p > 0 \, : \, P(p,0,H)> 0 \rbrace.
\end{equation*}
This is hence the lower boundary of the
strictly super-critical range of~$(p,H)$.
Some estimated values of the function~$p_c(H)$ are presented on Figure~\ref{fig.Phase-transiton-diagram}.

Finally, Figure~\ref{fig.Uc} presents some estimated values of  $U_c(p,H)$ for some selected parameters $(p,H)$ (mainly in the critical range).

\section{Proofs}
\label{s.Proofs}

\subsection{General approach}
Theorems~\ref{Thm.analogue:subcritical} to~\ref{Thm.permanently-super-critical} as well as Corollary~\ref{Coroll.phase-transition-hard-geometric} have been stated in terms of the dimensionless parameters $(p,U,H)$ introduced in Subsection~\ref{s.dimensionless-parameters}. This allowed us to introduce the different connectivity regimes and their frontiers in a more applied fashion, as illustrated by Figures~\ref{fig.Phase-transiton-diagram} and~\ref{fig.Phase-transiton-diagram-bis}. \\
\indent However, when working on the proofs of the aforementioned results, it will be much more convenient to come back to the original parameters $(\gamma, p, \lambda,r)$. This is mostly due to the fact that $H$, being inversely proportional to $r$, is less easy to work with when considering particular events related to connectivity in the network graph. Moreover, recalling that percolation of the connectivity graph $\mathcal{G}$ does not depend on the parameter $\gamma$, we can set $\gamma=1$ without loss of generality in what follows. Switching back to the original network parameters $(\gamma=1,p,\lambda,r)$, we will therefore now refer to the connectivity graph as $\mathcal{G}=\mathcal{G}_{p,\lambda,r}$  and prove the following equivalent formulations of Theorems~\ref{Thm.analogue:subcritical} to~\ref{Thm.permanently-super-critical}:

\begin{duplicate}[Reformulation of Theorem~\ref{Thm.analogue:subcritical}]
\label{reformulation-thm-subcritical}
For small enough $r \in \left(0,\infty\right]$ and small enough $\lambda>0$ (with the thresholds for $r$ and $\lambda$ not depending on one another), $\mathcal{G}_{1,\lambda,r}$ does not percolate and, consequently, $\mathcal{G}_{p,\lambda,r}$ does not percolate either for any $p \in \left[0,1\right]$.
\end{duplicate}

\begin{duplicate}[Reformulation of Theorem~\ref{Thm.analogue:supercritical}]
\label{reformulation-thm-supercritical}
For any $r \in \left(0,\infty\right]$, for large enough $p \in (0,1)$ and $\lambda < \infty$ (the thresholds for $p$ and $\lambda$ depend on $r$ but not on one another), $\mathcal{G}_{p,\lambda,r}$ percolates.
\end{duplicate}

\begin{duplicate}[Reformulation of Theorem~\ref{Thm.permanently-super-critical}]
\label{reformulation-thm-permanently-supercritical}
For large enough $p<1$ and large enough $r<\infty$ (with the thresholds for $p$ and $r$ not depending on one another), $\mathcal{G}_{p,0,r}$ percolates.
\end{duplicate}

The general idea of the proofs of the above reformulations will consist in using a \emph{coarse-graining argument}. This consists in proceeding as follows:
\begin{enumerate}
    \item Map the percolation of $\mathcal{G}$ to a discretized percolation process on a rescaled integer lattice.
    \item Depending on the needs, relate the percolation (or the absence of percolation) of $\mathcal{G}$ to the percolation (or absence of percolation) of the discretized process.
    \item Prove that the discretized process features only short-range dependencies. More precisely, we will refer to the notion of $k$-dependence, which will be introduced in due course, see Definition~\ref{def-k-dependence}.
    \item Apply a result from Liggett, Schonmann and Stacey (\cite[Theorem 0.0]{liggett_domination_1997}) to conclude that the discretized process is dominated by a Bernoulli percolation process and thereby draw conclusions on $\mathcal{G}$.
\end{enumerate}

\subsection{Preparation}
We begin with introducing a few notations and definitions that will be useful for the purposes of our developments. 

We will use the following convenient notation for the length  $\vert s \vert$  of a street segment $s$ (which is a connected, topologically closed
subset of some $e \in E$).
For $A \subset \mathbb{R}^{d}$ and $B \subset \mathbb{R}^{d}$, we denote as customary the Euclidean distance between $A$ and $B$ by:
 $$\text{dist}(A,B) \coloneqq \inf \lbrace \lVert x - y \rVert_{2} , \,  x \in A, \,  y \in B \rbrace.$$
For $x \in \mathbb{R}^{d}$, 
$a>0$ 
we denote by $Q_a(x) \coloneqq x + \left[-a/2,a/2\right]^{d}$ the $d$-dimensional cube of side $a$ centered at $x$. We note that this is exactly the definition of the closed ball $\mathscr{B}(x,a/2)$ with center $x$ and radius $a/2$ for the infinite norm of $\mathbb{R}^d$:
\begin{equation*}
    Q_a(x) = \lbrace y \in \mathbb{R}^d : \lVert y-x \rVert_{\infty} \leq a/2 \rbrace = \mathscr{B}(x,a/2).
\end{equation*}
For simplicity, whenever $a=n \in \mathbb{N} \setminus \lbrace 0 \rbrace$, we will write $Q_n$ to mean $Q_n(0)$. \\

\indent We denote by $\mathbf{M}$ the space of Borel measures on $\mathbb{R}^{d}$, equipped with the evaluation $\sigma$-algebra \cite[Section 13.1]{last2017lectures}, which is the smallest $\sigma$-algebra making the mappings $\mathbf{M} \ni \Xi \mapsto \Xi(B)$ measurable for all Borel sets $B \subset \mathbb{R}^{d}$. For a (possibly random) Borel measure $\mu$ on $\mathbb{R}^{d}$ and $A \subset \mathbb{R}^{d}$, we denote the restriction of $\mu$ to $A$ by $\mu_A(\cdot) \coloneqq \mu(A \cap \cdot)$. We also adapt the definition of the support of a measure as follows: let $\mu$ be a (possibly random) Borel measure on $\mathbb{R}^{2}$. The \emph{support} of $\mu$ is the following set:
$$\text{supp}(\mu) \coloneqq \lbrace x \in \mathbb{R}^{d} \, : \forall \epsilon > 0, \,  \mu(Q_{\epsilon}(x)) > 0 \rbrace.$$

We will also need the concepts of \emph{stabilization} and \emph{asymptotic essential connectedness}, both introduced in \cite{hirsch2018continuum} for investigating spatial dependencies of random measures. 

\begin{definition}
\label{Def.stabilizing}
\cite[Definition 2.3]{hirsch2018continuum}
A random measure $\Xi$ on $\mathbb{R}^{d}$ is called \emph{stabilizing} if there exists a random field of stabilization radii $R = \lbrace R_{x} \rbrace _{x \in \mathbb{R}^{d}}$  defined on the same probability space as $\Xi$
and $\Xi$-measurable, 
such that:
\begin{enumerate}[label = (\arabic*)]
\item $(\Xi,R)$ are jointly stationary,
\item $\displaystyle \lim_{n \uparrow \infty} \mathbb{P}\left(\sup _{y \in Q_{n} \cap \mathbb{Q}^{d}} R_{y} < n\right) = 1$,
\item for all $n \geq 1$, the random variables 
$$ \left\lbrace f\left(\Xi_{Q_{n}(x)}\right)\mathbbm{1}\left\lbrace \sup_{y \in Q_{n}(x) \cap \mathbb{Q}^{d}} R_{y} < n \right\rbrace \right\rbrace _{x \in \varphi}$$
are independent for all bounded measurable functions $$f : \mathbf{M} \to \left[0, +\infty \right)$$ and finite $\varphi \subset \mathbb{R}^{d}$ such that $\forall \, x \in \varphi, \, \text{dist}(x, \varphi \setminus \lbrace x \rbrace ) > 3n$.
\end{enumerate}
\end{definition}

\indent We slightly modify the definition of asymptotic essential connectedness given in~\cite{hirsch2018continuum} for the sake of simplicity and use the following definition:
\begin{definition}
\label{Def.eac}
Let $\Xi$ be a random measure on $\R^d$. Then $\Xi$ is \emph{asymptotically essentially  connected} if there exists a random field $R = \lbrace R_{x} \rbrace_{x \in \mathbb{R}^{d}}$ such that $\Xi$ is stabilizing with  $R$ as in Definition~\ref{Def.stabilizing} and if for all $n \geq 1$, whenever $\displaystyle \sup_{y \in Q_{2n} \cap \mathbb{Q}^{d}} R_{y} < n/2$, the following assertions are satisfied:
\begin{enumerate}[label = (\arabic*)]
\item $\text{supp}(\Xi_{Q_{n}}) \neq \emptyset$,
\item $\text{supp}(\Xi_{Q_{n}})$ is contained in a connected component of $\text{supp}(\Xi_{Q_{2n}})$.
\end{enumerate}
\end{definition}

The following result is stated in~\cite[Example 3.1]{hirsch2018continuum} for a slightly modified version of Definition~\ref{Def.eac}. It is easy to check that it adapts in our case as follows:

\begin{proposition}\label{prop.PVT-stabilizes}
Let $\Lambda = \nu_{1}(S \cap dx)$, where $S$ is the PVT generated by an homogeneous stationary Poisson point process. Then $\Lambda$ is stabilizing and asymptotically essentially connected with  the following stabilization field:
$$\forall x \in \mathbb{R}^{d}, \, R_x := \inf \lbrace \lVert x - X_{S,i} \rVert_2 , \, X_{S,i} \in X_{S} \rbrace,$$
where $X_S$ is the PPP  generating $S$.
\end{proposition}

For simplicity, for $x \in \mathbb{R}^{d}$ and $n \in \mathbb{N}\setminus \lbrace 0 \rbrace$ we denote:
$$R(Q_n(x)) \coloneqq \sup_{y \in Q_n(x) \cap \mathbb{Q}^{d}} R_y.$$

Finally,  we define the openness and closedness of crossroads and street segments (possibly the whole streets themselves) as follows:

\begin{definition}[Open/Closed crossroad]
\label{Def.open/crossroad}
We say a crossroad $v \in V$ is \emph{open} if it is an atom of the point process $Y$, i.e. $Y(\lbrace v\rbrace) = 1$. We say $v$ is \emph{closed} if it is not open.
\end{definition}

\begin{definition}[Open/Closed street segment]
\label{Def.open/closed/subcritical}
Let $e \in E$ be a street and let $ s \subseteq e$ be a non-empty street segment. We say $s$ is \emph{open} if either of the two following set of conditions are satisfied:
\begin{enumerate}
\item $\vert s \vert \leq r$
\vspace{.2 cm}
\item[]\textbf{OR}
\vspace{.2 cm}
\item $\left\{
\begin{array}{l}
\vert s \vert > r \\
\forall c \subset s, \, (\vert c \vert = r \; , c \text{ connected and} \; c  \; \text{topologically closed} )\Rightarrow X^{\lambda}(c) \geq 1.
\end{array}
\right.$
\end{enumerate}
We say that $s$ is \emph{closed} if $s$ is not open. \\

\noindent We are now ready to proceed with the proofs of Theorems~\ref{thm-ergodicity}-\ref{Thm.permanently-super-critical}.
\end{definition}

\subsection{Proof of Proposition~\ref{thm-ergodicity}}
To prove that $Z$ is mixing, we will work on the canonical space and it will thus suffice to show that 
\begin{equation}
\label{eq-definition-mixing}
    \lim_{\lVert x \rVert_2 \rightarrow \infty} \mathbb{P}(A \cap S_xB) = \mathbb{P}(A)\mathbb{P}(B),
\end{equation}
for all events $A,B \in \sigma(Z)$ that are measurable with respect to the sigma-algebra $\sigma(Z)$ generated by $Z$ and where $\lbrace S_x \rbrace_{x \in \mathbb{R}^d}$ denotes the natural shift on $\mathbb{R}^d$. \\

\indent Note first that by~\cite[Lemma 12.3.II]{daley2008introduction} (and as has been done in the proof of \cite[Proposition 12.3.VI]{daley2008introduction}),  it suffices to check the mixing condition~\eqref{eq-definition-mixing} for \emph{local events}, i.e. of the form $A \in \sigma(Z \cap W_A)$ and $B \in \sigma(Z \cap W_B)$, where $W_A$ and $W_B$ are compact observation windows in $\mathbb{R}
^d$. Thus, let $A$ and $B$ be such events. We will show that for all $\epsilon >0$, we can find $x$ with $\lVert x \rVert_2$ sufficiently large so that $ \vert \mathbb{P}(A \cap S_xB) - \mathbb{P}(A)\mathbb{P}(B) \vert \leq \epsilon$. \\

\indent Take any $\epsilon >0$. By condition~(2) in the definition of stabilization (Definition~\ref{Def.stabilizing}), we can find sufficiently large $n \geq 1$ such that $\mathbb{P}(R(Q_n) \geq n)\leq \epsilon/3$. Moreover, such $n$ can be chosen so as to satisfy $W_A \subseteq Q_n$ and $W_B \subseteq Q_n$. Fix such $n$. \\
\indent Since $A \in \sigma(Z \cap W_A)$, $A$ only depends on the configuration of $Z$ inside $W_A$. In the same way, $B$ only depends on the configuration of $Z$ inside $W_B$, and $S_xB$ only depends on the configuration of $Z$ inside $S_xW_B = W_B -x$. Since $B \subseteq Q_n$, we have that $S_xW_B \subseteq Q_n -x =: Q_n(-x)$. Take $x$ with $\lVert x \rVert_2 > 6n\sqrt{2}$. Then we have $Q_n \cap Q_n(-x) = \emptyset$ and thus $W_A \cap S_xW_B = \emptyset$, so that the events $A$ and $S_xB$ depend on the configuration of $Z$ in disjoint sets. Since the conditional distribution of $Z$ given the random support $\Lambda$ is that of a superposition of a Bernoulli process and of a Poisson point process, the events $A$ and $S_xB$ are conditionnally independent given $\Lambda$. Hence:

\begin{equation}
    \label{eq-ergodicity-using-independence}
    \mathbb{P}(A \cap S_xB) = \mathbb{E} \Big[\mathbb{E}\left(\mathbbm{1} \lbrace A \rbrace  \mathbbm{1} \lbrace S_xB \rbrace \, \vert \, \Lambda \right) \Big] = \mathbb{E} \Big[\mathbb{E} (\mathbbm{1} \lbrace A \rbrace \, \vert \, \Lambda) \mathbb{E} \left(  \mathbbm{1} \lbrace S_xB \rbrace \, \vert \, \Lambda \right) \Big].
\end{equation}

Now, since $A \in \sigma(Z \cap W_A)$ with $W_A \subseteq Q_n$, we can write $\mathbb{E}(\mathbbm{1} \lbrace A \rbrace \, \vert \Lambda ) = f(\Lambda_{Q_n})$ as a bounded deterministic function of $\Lambda_{Q_n}$. In the same way, we can write $\mathbb{E}(\mathbbm{1} \lbrace S_xB \rbrace \, \vert \Lambda ) = g(\Lambda_{Q_n(-x)})$ as a bounded deterministic function of $\Lambda_{Q_n(-x)}$. By~\eqref{eq-ergodicity-using-independence}, we thus get:

\begin{align}
    \notag \mathbb{P}(A \cap S_x B) &= \mathbb{E}\left[ f(\Lambda_{Q_n})g(\Lambda_{Q_n(-x})\right] \\
    \label{eq-ergodicity-introducing-R}
    \begin{split}
       &=\mathbb{E}\left[ f(\Lambda_{Q_n})g(\Lambda_{Q_n(-x)}) \mathbbm{1} \lbrace R(Q_n) < n \rbrace \mathbbm{1} \lbrace R(Q_n(-x)) < n \rbrace \right] \\
       &+\mathbb{E}\left[ f(\Lambda_{Q_n})g(\Lambda_{Q_n(-x)}) \mathbbm{1} \lbrace (R(Q_n) \geq n) \cup ( R(Q_n(-x)) \geq n) \rbrace \right].
    \end{split}
\end{align}

Let us first deal with the second term appearing in the right-hand side of~\eqref{eq-ergodicity-introducing-R}. Using the fact that both $f$ and $g$, being conditional expectations of indicator functions, are upper-bounded by $1$, we get:

\begin{gather}
    \notag \mathbb{E}\left[ f(\Lambda_{Q_n})g(\Lambda_{Q_n(-x)}) \mathbbm{1} \lbrace (R(Q_n) \geq n) \cup ( R(Q_n(-x)) \geq n) \rbrace \right] \\
    \notag \leq \mathbb{P}\left[ (R(Q_n) \geq n) \cup ( R(Q_n(-x)) \geq n) \right] \\
    \label{eq-ergodicity-union-bound}
    \leq \mathbb{P}\left[R(Q_n) \geq n\right] + \mathbb{P}\left[R(Q_n(-x)) \geq n\right],
\end{gather}
where we have used the union bound in~\eqref{eq-ergodicity-union-bound}. Now, by stationarity of the stabilization field $\lbrace R_y \rbrace_{y \in \mathbb{R}^d}$, we get that the right-hand side in~\eqref{eq-ergodicity-union-bound} is equal to $2\mathbb{P}\left[R(Q_n) \geq n\right]$. In all, we thus get:

\begin{equation}
    \label{eq-ergodicity-dealing-with-second-term}
    \mathbb{E}\left[ f(\Lambda_{Q_n})g(\Lambda_{Q_n(-x}) \mathbbm{1} \lbrace (R(Q_n) \geq n) \cup ( R(Q_n(-x)) \geq n) \rbrace \right] \leq 2\epsilon/3.
\end{equation}

We now deal with the first term appearing in the right-hand side of~\eqref{eq-ergodicity-introducing-R}. Note that since $\lVert x \rVert_2 > 6n\sqrt{2}$, the set $\varphi := \lbrace 0,-x \rbrace \subset \mathbb{R}^d$ satisfies $\forall y \in \varphi, \text{dist}(y,\phi\setminus \lbrace y \rbrace) > 3n$ and so, by the condition (3) in the definition of stabilization, the random variables $f(\Lambda_{Q_n})\mathbbm{1}\lbrace R(Q_n) <n \rbrace$ and $g(\Lambda_{Q_n(-x)})\mathbbm{1}\lbrace R(Q_n(-x)) <n \rbrace$ are independent. Thus, the first term appearing in the right-hand side of~\eqref{eq-ergodicity-introducing-R} becomes:

\begin{equation}
\label{eq-ergodicity-using-stab-first-term}
\begin{gathered}
      \mathbb{E}\left[ f(\Lambda_{Q_n})g(\Lambda_{Q_n(-x)}) \mathbbm{1} \lbrace R(Q_n) < n \rbrace \mathbbm{1} \lbrace R(Q_n(-x)) < n \rbrace \right]  \\ = \mathbb{E}\left[ f(\Lambda_{Q_n}) \mathbbm{1} \lbrace R(Q_n) < n \rbrace\right]\mathbb{E}\left[ g(\Lambda_{Q_n(-x)})  \mathbbm{1} \lbrace R(Q_n(-x)) < n \rbrace \right]. 
\end{gathered}
\end{equation}

Now, using the fact that $f(\Lambda_{Q_n})=:\mathbb{E}\left( \mathbbm{1} \lbrace A \rbrace \, \vert \, \Lambda \right)$ and noting that the event $\lbrace R(Q_n) < n \rbrace$ is $\Lambda$-measurable, we can put everything back into a single expectation and get:

\begin{align*}
    \mathbb{E}\left[ f(\Lambda_{Q_n}) \mathbbm{1} \lbrace R(Q_n) < n \rbrace\right] &= \mathbb{E}\left[ \mathbb{E}\left( \mathbbm{1} \lbrace A \rbrace \, \vert \, \Lambda \right) \mathbbm{1} \lbrace R(Q_n) < n \rbrace\right]  \\ &= \mathbb{E}\left[ \mathbb{E}\left( \mathbbm{1} \lbrace A \rbrace \mathbbm{1} \lbrace R(Q_n) < n \rbrace \, \vert \, \Lambda \right)\right] \\
 &= \mathbb{E}\left[  \mathbbm{1} \lbrace A \rbrace \mathbbm{1} \lbrace R(Q_n) < n \rbrace \right]  \\ &= \mathbb{P}(A \cap \lbrace  R(Q_n) < n \rbrace ).
\end{align*}

In the same way, we get: 
$$\mathbb{E}\left[ g(\Lambda_{Q_n(-x)}) \mathbbm{1} \lbrace R(Q_n(-x)) < n \rbrace\right] = \mathbb{P}(S_xB \cap \lbrace  R(Q_n(-x)) < n \rbrace ).$$ Thus,~\eqref{eq-ergodicity-using-stab-first-term} yields:

\begin{equation*}
\begin{gathered}
      \mathbb{E}\left[ f(\Lambda_{Q_n})g(\Lambda_{Q_n(-x)}) \mathbbm{1} \lbrace R(Q_n) < n \rbrace \mathbbm{1} \lbrace R(Q_n(-x)) < n \rbrace \right]  \\ =\mathbb{P}(A \cap \lbrace  R(Q_n) < n \rbrace )\mathbb{P}(S_xB \cap \lbrace  R(Q_n(-x)) < n \rbrace ) \\
       =\mathbb{P}(A \cap \lbrace  R(Q_n) < n \rbrace )\mathbb{P}\circ S_x(B \cap \lbrace  R(Q_n) < n \rbrace ) \\ = \mathbb{P}(A \cap \lbrace  R(Q_n) < n \rbrace )\mathbb{P}(B \cap \lbrace  R(Q_n) < n \rbrace ),
\end{gathered}
\end{equation*}
where we have used the stationarity assumption to get the last line. Finally, using the fact that $\mathbb{P}(R(Q_n)<n)\geq 1-\epsilon/3$, we get 
$$\vert \mathbb{P}(A \cap \lbrace  R(Q_n) < n \rbrace )\mathbb{P}(B \cap \lbrace  R(Q_n) < n \rbrace ) - \mathbb{P}(A)\mathbb{P}(B) \vert \leq \epsilon/3$$
and thus 

\begin{equation}
    \label{eq-ergodicity-bounding-first-term}
     \vert \mathbb{E}\left[ f(\Lambda_{Q_n})g(\Lambda_{Q_n(-x)}) \mathbbm{1} \lbrace R(Q_n) < n \rbrace \mathbbm{1} \lbrace R(Q_n(-x)) < n \rbrace \right] - \mathbb{P}(A)\mathbb{P}(B) \vert \leq \epsilon/3.
\end{equation}

Using~\eqref{eq-ergodicity-dealing-with-second-term},~\eqref{eq-ergodicity-bounding-first-term} to put everything back together in~\eqref{eq-ergodicity-introducing-R} and using the triangular inequality, we finally get:

\begin{equation*}
    \lvert \mathbb{P}(A \cap S_xB) - \mathbb{P}(A)\mathbb{P}(B) \rvert \leq 2\epsilon/3 + \epsilon/3 = \epsilon
\end{equation*}
for sufficiently large $x$, as required. This concludes the proof of Proposition~\ref{thm-ergodicity}.

\begin{remark}
Note that we actually did not need to use the PVT structure nor the asymptotic essential connectedness of the random support $S$ here. We only used the fact that $S$ is a stabilizing random tessellation and the complete independence properties of the point processes of users $X^\lambda$ and of relays $Y$ given their random support $S$. As a matter of fact, Proposition~\ref{thm-ergodicity} can be generalized to any stabilizing random tessellation $S$ in $\mathbb{R}
^d$.
\end{remark}

\subsection{Proof of Theorem~\ref{Thm.analogue:subcritical}}
As mentioned earlier, proving Theorem~\ref{Thm.analogue:subcritical} is equivalent to proving Reformulation~\ref{reformulation-thm-subcritical}. This in turn is equivalent to showing that $\mathcal{G}$ does not percolate when $p=1$ and  $\lambda, r$  are sufficiently small but positive. 
We will use a coarse-graining argument and introduce a discrete  site percolation model on the integer lattice $\mathbb{Z}^d$ constructed in such a way that if it does not percolate, then neither does $\mathcal{G}$. Proving the absence of percolation of the integer lattice model will then be done via appealing to its local dependence.

To this end, for $n \geq 1$, say a site $z \in \mathbb{Z}^{d}$ is \emph{n-good} if the following conditions are satisfied: 
\begin{enumerate}[label = (\arabic*)]
\item $R(Q_n(nz)) < n$,
\item $\forall e \in E$ if  $s_{z,e}  \coloneqq e \cap Q_n(nz)\not=\emptyset$,  then $s_{z,e}$ is closed.
\end{enumerate}
Say a site $z \in \mathbb{Z}^{d}$ is \emph{$n$-bad} if it is not $n$-good.
\\

Our first claim is the following:
\begin{lemma}
\label{Claim1.subcritical}
Percolation of $\mathcal{G}$ implies percolation of the process of $n$-bad sites.
\end{lemma}
\begin{proof}
Assume $\mathcal{G}$ percolates and denote by $\mathcal{C}$ an  unbounded (connected) component of $\mathcal{G}$.
 Denote $\mathcal{Z}=\mathcal{Z}_n:=\lbrace z\in \mathbb{Z}^d: \mathcal{C} \cap Q_n(nz) \neq \emptyset \rbrace$.  Since $\mathcal{C}$ is unbounded we have $\#(\mathcal{Z})=\infty$.
 Observe for all $z\in\mathcal{Z}$, $z$ is $n$-bad since  condition~(2) of $n$-goodness is not satisfied (there exists an open street segment intersecting~$Q_n(nz)$). 
 Also, $\mathcal{Z}$ is almost surely connected in $\mathbb{Z}^d$ (in the following sense: $z,z'\in\mathbb{Z}^d$, $z\not=z'$ are connected in~$\mathbb{Z}^d$ if $\lVert z-z' \rVert_{1}=1$).
This follows from the fact that the probability that some edge $e\in E$ of the PVT intersects $\mathbb{Z}^d$  is equal to zero (which is actually also true for the Voronoi tessellation generated by any stationary point process, as a consequence of the fact that such a process does not have points which are equidistant to a given, fixed
location, see e.g.~\cite[Lemma 11.2.3]{blaszczyszyn2017lecture}).
Hence, the process of $n$-bad sites percolates. 
\end{proof}

By Lemma~\ref{Claim1.subcritical}, it suffices to prove that the process of $n$-bad sites does not percolate (for some $n$) when  $\lambda$ and $r$ are sufficiently small but positive.
This will be done using the fact that it is a $3$-dependent  percolation model
on the integer lattice $\mathbb{Z}^d$.
\begin{definition}
\label{def-k-dependence}
Let $ \mathbf{X}=(X_z)_{z \in \mathbb{Z}^{d}}$ be a discrete random field. Let $k \geq 1$. Then $\mathbf{X}$ is said to be $k$-dependent if for all $q \geq 1$ and all $ \lbrace s_{1}, \ldots s_{q} \rbrace \subset \mathbb{Z}^{d}$ finite with the property that $\forall i \neq j, \lVert s_{i} - s_{j} \rVert_{\infty} > k $, the random variables $(X_{s_{i}})_{1 \leq i \leq q}$ are independent.
\end{definition}

As previously stated, we have the following:

\begin{lemma}
\label{Claim2.subcritical}
For $z \in \mathbb{Z}^{d}$, set $\zeta_{z} \coloneqq \mathbbm{1}\lbrace z \, \text{is $n$-bad} \rbrace$. Then $(\zeta_{z})_{z \in \mathbb{Z}^{d}}$ is a $3$-dependent random field.
\end{lemma}
\begin{proof}
As a starting point, note that $\forall z \in \mathbb{Z}^{d}, \zeta_{z} = 1 - \mathbbm{1} \lbrace z \, \text{is $n$-good} \rbrace$. It is therefore equivalent to prove that the process of $n$-good sites is $3$-dependent. \\
\indent For $z \in \mathbb{Z}^{d}$, set $\xi_{z} = \mathbbm{1} \lbrace z \, \text{is $n$-good} \rbrace$. Let $ \lbrace z_{1}, \ldots z_{q} \rbrace \subset \mathbb{Z}^{d}$ be such that $\forall i \neq j, \lVert z_i - z_j \rVert_{\infty} > 3$.
We want to show that the random variables $(\xi_{z_i})_{1\leq i \leq q}$ are independent. Since we are dealing with indicator functions, this is equivalent to showing that: 
$$\mathbb{E}\left(\prod_{i=1}^{q} \xi_{z_i} \right) = \prod_{i=1}^{q} \mathbb{E}(\xi_{z_i}).$$
Now, we have:

\begin{align}
  \nonumber \mathbb{E}\left(\prod_{i=1}^{q} \xi_{z_i} \right) &= \mathbb{E}\left[\mathbb{E}\left(\prod_{i=1}^{q} \xi_{z_i} \Bigg\vert \Lambda \right)\right] \\ \nonumber &= \mathbb{E}\Bigg[\mathbb{E}\Bigg(\prod_{i=1}^{q}\mathbbm{1}\lbrace R(Q_n(nz_i)) < n \rbrace \prod_{i=1}^{q} \mathbbm{1} \lbrace \forall \, e \in E, \, s_{z_i,e} \, \text{is closed} \rbrace  \Bigg\vert \Lambda \Bigg) \Bigg] \\  &= \label{eq3}  \mathbb{E}\Bigg[\prod_{i=1}^{q} \mathbbm{1} \lbrace R(Q_n(nz_i)) <n \rbrace \mathbb{E}\Bigg( \prod_{i=1}^{q} \mathbbm{1} \lbrace \forall \, e \in E, \, s_{z_i,e} \, \text{is closed} \rbrace  \Bigg\vert \Lambda \Bigg) \Bigg],
\end{align}
where we have used $\Lambda$-measurability of the random variables $\lbrace R_x \rbrace_{x \in \mathbb{R}^d}$ in~\eqref{eq3}. \\

For $1 \leq i \leq q$, set $A_{z_i} \coloneqq \lbrace \forall \, e \in E, \, s_{z_i,e} \, \text{is closed} \rbrace$. According to Definition~\ref{Def.open/closed/subcritical}, for a given $1 \leq i \leq q$, the event $A_{z_i}$ only depends on the configuration of the random measure $\Lambda$ and of the Cox point process $X^{\lambda}$ inside the $d$-dimensional cube $Q_n(nz_i)$. Therefore, given $\Lambda$, the events $\lbrace A_{z_i}  : 1 \leq i \leq q \rbrace$ only depend on $X^{\lambda} \cap Q_n(nz_i)$, $1 \leq i \leq q$. Since we have $\forall i \neq j, \lVert z_i - z_j \rVert_{\infty} > 3$, then the $d$-dimensional cubes $Q_n(nz_i)$ are disjoint. Moreover, given $\Lambda$, $X^{\lambda}$ has the distribution of a Poisson point process. Thus, by Poisson independence property, the events $(A_{z_i})_{1 \leq i \leq q}$ are conditionally independent given $\Lambda$. Hence~\eqref{eq3} yields:
\begin{align}
 \label{eq4} \mathbb{E}\left(\prod_{i=1}^{q} \xi_{z_i} \right) &= \mathbb{E}\Bigg[\prod_{i=1}^{q} \mathbbm{1} \lbrace R(Q_n(nz_i)) <n \rbrace \prod_{i=1}^{q} \mathbb{E}\Bigg(  \mathbbm{1} \lbrace \forall \, e \in E, \, s_{z_i,e} \, \text{is closed} \rbrace  \Bigg\vert \Lambda \Bigg) \Bigg]. 
\end{align}
Recall that the restriction of the random measure $\Lambda$ to the $d$-dimensional cube $Q_n(x)$ is denoted by $\Lambda_{Q_n(x)} (\cdot) =: \Lambda(Q_n(x) \cap \, \cdot)$ and set $f(\Lambda_{Q_n(x)}) \coloneqq \mathbb{E}\Bigg(  \mathbbm{1} \lbrace \forall \, e \in E, \, s_{x,e} \, \text{is closed} \rbrace  \Bigg\vert \Lambda \Bigg)$. Then $f$ is a deterministic, bounded and measurable function of $\Lambda_{Q_n(x)}$. Moreover, the set $\varphi \coloneqq \lbrace nz_1,\ldots,nz_q \rbrace \subset \mathbb{R}^{d}$ is a finite subset of $\mathbb{R}^d$ satisfying: $$\forall i \neq j, \lVert nz_i - nz_j \rVert_{\infty} > 3n.$$ Since the infinite norm is always upper bounded by the Euclidean norm, we have $\forall i \neq j, \lVert nz_i - nz_j \rVert_{2} > 3n $, and so $\varphi$ satisfies:
$$\forall x \in \varphi, \, \text{dist}(x, \varphi \setminus \lbrace x \rbrace) > 3n.$$
Hence, by condition (3) in the definition of stabilization (Definition~\ref{Def.stabilizing}), the random variables appearing in the right-hand side of~\eqref{eq4} are independent. This yields:
\begin{align*}
     \mathbb{E}\left(\prod_{i=1}^{q} \xi_{z_i} \right) &= \prod_{i=1}^{q} \mathbb{E}\Bigg(\mathbbm{1} \lbrace R(Q_n(nz_i)) <n \rbrace \prod_{i=1}^{q} \mathbb{E}\Bigg(  \mathbbm{1} \lbrace \forall \, e \in E, \, s_{z_i,e} \, \text{is closed} \rbrace  \Bigg\vert \Lambda \Bigg) \Bigg) \\
     &=\prod_{i=1}^{q} \mathbb{E}\left(\xi_{z_i} \right),
     \end{align*}
     
thus concluding the proof of the lemma.
\end{proof}
Now we prove that the probability for an arbitrary site, which by stationarity can be chosen to be the origin $\boldsymbol{0} \in \mathbb{Z}^{d}$, to be $n$-bad can be made arbitrarily
small  when first  taking  some large enough finite $n$
and then positive small enough $\lambda,r$,
as stated in the following Lemma.
\begin{lemma}\label{Claim3.subcritical}
\begin{equation*}
    \lim_{n \uparrow \infty}\lim_{\lambda, r \downarrow 0} \mathbb{P}(\boldsymbol{0} \text{ is $n$-bad}) = 0.
\end{equation*}
\end{lemma}
\begin{proof}
Note that we have:
\begin{align*}
    \mathbb{P}(\boldsymbol{0} \, \text{is $n$-bad}) &= \mathbb{P}\Big(\lbrace R(Q_n) \geq n \rbrace \cup \lbrace \exists \, e \in E:\,  e \cap Q_n\not=\emptyset \text{ and  open}  \rbrace \Big) \\ & \leq \mathbb{P}\left( R(Q_n) \geq n \right) + \mathbb{P}\left(  \exists \, e \in E:\,  e \cap Q_n\not=\emptyset \text{ and  open} \right)\\
&\le
\mathbb{P}\left( R(Q_n) \geq n \right) \tag{a}\\
&\hspace{1em}
+ \mathbb{P}\left(  \exists \, e \in E:\ 0<|e \cap Q_n|\le r \right)\tag{b}\\
&\hspace{1em}
+ \mathbb{P}\left(  \exists \, e \in E:\ e \cap Q_n \text{ satisfies cond.~(2) in Definition~\ref{Def.open/closed/subcritical}}\right).\tag{c}
\end{align*}
Take any $\epsilon>0$. By the stabilization property of the PVT (Proposition~\ref{prop.PVT-stabilizes})  we have $\lim_{n \uparrow \infty} \mathbb{P}(R(Q_n) \geq n)  = 0$ and so we can fix   $n$ large enough to make the probability in~(a) smaller than $\epsilon/3$. 
Then, $Q_n$ intersects  almost surely zero or a finite number of edges $e\in E$. Hence the probability in~(b) converges to $0$ when $r\to 0$ and, consequently, we can take $r$ small enough to make the probability in~(b)  smaller  than $\epsilon/3$.
Finally, for given $n$ (and independently of $r$), we can take $\lambda$ small enough to make the probability in~(c) smaller than $\epsilon/3$. Indeed, this latter probability is dominated by the  probability that $X^\lambda(Q_n)\ge 1$ and thus converges to $0$ when $\lambda\to 0$ for any finite $n$.
This concludes the proof of Lemma~\ref{Claim3.subcritical}.
\end{proof}

By Lemmas~\ref{Claim2.subcritical} and~\ref{Claim3.subcritical}, using~\cite[Theorem 0.0]{liggett_domination_1997}, for large enough $n<\infty$ and small enough $r>0,\lambda>0$, the process of $n$-bad sites is stochastically dominated from above by an independent site percolation model on the integer lattice where the probability of having an open site is arbitrarily small. Hence this independent site percolation model is sub-critical. Consequently, we can make the process of $n$-bad sites  non-percolating. By~Lemma~\ref{Claim1.subcritical}
the same is true for $\mathcal{G}$, thus concluding the proof of Theorem~\ref{Thm.analogue:subcritical}.

\subsection{Proof of Theorem~\ref{Thm.analogue:supercritical}}
We shall prove that for large enough $p<1$ the model  $\mathcal{G}$ percolates with positive probability  
for all  $r>0$ and large enough 
$\lambda<\infty$ (depending on $r$).

As in the proof of Theorem~\ref{Thm.analogue:subcritical}, we will use a coarse-graining argument.
To this end, consider the following percolation model on the integer lattice $\mathbb{Z}^d$. For $n \geq 1$, say a site $z \in \mathbb{Z}^{d}$ is $n$-good if the following conditions are satisfied:
\begin{enumerate}[label = (\arabic*)]
\item $R(Q_{6n}(nz)) < 6n$. 
\item There exists a \emph{full} street (not just a street segment!) entirely contained in $Q_n(nz)$. By abuse of notation, we denote this by $E \cap Q_n(nz) \neq \emptyset$. 
\item There exists $e \in E \cap Q_n(nz)$ such that $e$ is open, in the sense of Definition~\ref{Def.open/closed/subcritical}. In other words, there exists an open street which is fully included in the cube $Q_n(nz)$.
\item All crossroads in $Q_{6n}(nz)$ are open, in the sense of Definition~\ref{Def.open/crossroad}.
\item Every two open edges $e,e' \in E \cap Q_{3n}(nz)$ are connected by a path in $\mathcal{G}\cap Q_{6n}(nz)$.
\end{enumerate}
We say a site $z \in \mathbb{Z}^{d}$ is $n$-bad if it is not $n$-good. \\

The $n$-good sites have been defined so as to satisfy the following implication.
\begin{lemma}
\label{Claim1.supercritical}
Percolation of the process of $n$-good sites implies percolation of the connectivity graph $\mathcal{G}$.
\end{lemma}
\begin{proof}
Let $\mathcal{C}$ be an infinite connected component of $n$-good sites. Consider $z,z'\in \mathcal{C}$ 
such that $\lVert z-z' \rVert_1=1$.
Without loss of generality, assume
 $z = (a_1,a_2, \ldots, a_d)$ for some $a_1, a_2, \ldots, a_d \in \mathbb{Z}$ and  $z' = (a_1+1,a_2, \ldots, a_d)$.
 By condition (2)  in the definition of $n$-goodness, we can find  open
 edges $e\in E \cap Q_n(nz)$
 and $e'\in E \cap Q_n(nz')$.  Since
\begin{gather*}
Q_n(nz) = \left[na_1-n/2,na_1+n/2\right] \times \ldots \times \left[na_d-n/2,na_d+n/2\right], \\ Q_n(nz') = \left[na_1+n/2,na_1+3n/2\right] \times \left[na_2-n/2,na_2+n/2\right] \times \ldots \times \left[na_d-n/2,na_d+n/2\right], \\
Q_{3n}(nz) = \left[na_1-3n/2,na_1+3n/2\right] \times \ldots \times  \left[na_d-3n/2,na_d+3n/2\right], \\
Q_{6n}(nz) = \left[na_1-3n,na_1+3n\right] \times \ldots \times  \left[na_d-3n,na_d+3n\right],
\end{gather*}
we have $Q_{n}(nz') \subset Q_{3n}(nz)$ and so $e' \in E \cap Q_{n}(nz')$ implies $e' \in E \cap Q_{3n}(nz)$. Since we also have $e \in E \cap Q_{n}(nz) \subset  E \cap Q_{3n}(nz)$ and $e,e'$ are both open, by conditions (4) and  (5) in the definition of $n$-goodness, $e$ and $e'$ are connected by a path $\mathcal{L}$ in $\mathcal{G}\cap Q_{6n}(nz)$. Therefore, the path $\mathcal{L}$ also connects $e$ and $e'$ in $\mathcal{G}$, thus giving rise to an infinite connected component in $\mathcal{G}$.  This concludes the proof of Lemma~\ref{Claim1.supercritical}.
\end{proof}

\begin{lemma}
\label{Claim2.supercritical}
For $z \in \mathbb{Z}^{d}$, set $\xi_{z} \coloneqq \mathbbm{1}\lbrace z \, \text{is $n$-good} \rbrace$. Then $(\xi_{z})_{z \in \mathbb{Z}^{d}}$ is an $18$-dependent random field.
\end{lemma}
\begin{proof}
In the same way as in the proof of Lemma~\ref{Claim2.subcritical}, it suffices to prove that for all finite $\psi = \lbrace z_{1}, \ldots z_{q} \rbrace \subset \mathbb{Z}^{d}$ such that $\forall i \neq j, \lVert z_i - z_j \rVert_{\infty} > 18$, we have:
$$\mathbb{E}\left(\prod_{i=1}^{q} \xi_{z_i} \right) = \prod_{i=1}^{q} \mathbb{E}(\xi_{z_i}).$$
Denote respectively by $A_z,B_z,C_z,D_z,F_z$ the events that the conditions (1), (2), (3), (4), (5) in the definition of $n$-goodness hold for $z \in \mathbb{Z}^{d}$. We thus have: 
$$\forall z \in \mathbb{Z}^{d}, \, \xi_{z} = \mathbbm{1}\lbrace A_z \rbrace \mathbbm{1}\lbrace B_z \rbrace \mathbbm{1}\lbrace C_z \rbrace \mathbbm{1}\lbrace D_z \rbrace \mathbbm{1}\lbrace F_z \rbrace.$$

Note first that whenever $z \in \mathbb{Z}^{d}$, the indicators $\mathbbm{1} \lbrace A_z \rbrace$ and $\mathbbm{1} \lbrace B_z \rbrace$ are $\Lambda$-measurable. Thus, we have :
\begin{align}
\nonumber \mathbb{E}\left(\prod_{i=1}^{q} \xi_{z_i} \right) &= \mathbb{E}\left[\mathbb{E}\left(\prod_{i=1}^{q} \xi_{z_i} \Bigg\vert \Lambda \right)\right] 
\\ 
\nonumber &= \mathbb{E}\left[\prod_{i=1}^{q} \mathbbm{1}\lbrace A_{z_i} \rbrace \mathbbm{1}\lbrace B_{z_i}\rbrace\mathbb{E}\left(  \prod_{i=1}^{q}\mathbbm{1}\lbrace C_{z_i} \cap D_{z_i} \cap F_{z_i}  \rbrace \Bigg\vert \Lambda \right)\right].
\end{align}
Now, note that conditioned on $\Lambda$, for each $1 \leq i \leq q$, the event $ C_{z_i} \cap D_{z_i} \cap F_{z_i} $ only depends on the configuration of $X^{\lambda}$ and  $Y$ inside of the $d$-dimensional cube $Q_{6n}(nz_i)$.
Since $\psi$ satisfies $\forall i \neq j, \lVert z_i - z_j \rVert_{\infty} > 18$, then we have $\forall i \neq j, \lVert nz_i - nz_j \rVert_{\infty} > 18n$. As a matter of fact, the cubes $\lbrace Q_{6n}(nz_i)) \, : \, 1 \leq i \leq q\rbrace$ are disjoint, i.e. $$\forall i \neq j, Q_{6n}(nz_i) \cap Q_{6n}(nz_j) = \emptyset.$$
By the complete independence of Poisson and Bernoulli processes (recall that, given $\Lambda$, $X^\lambda$ has the distribution of a Poisson point process and $Y$ the distribution of a Bernoulli point process), we have:

\begin{align}
\nonumber \mathbb{E}\left(\prod_{i=1}^{q} \xi_{z_i} \right)  &= \mathbb{E}\left[\prod_{i=1}^{q} \mathbbm{1}\lbrace A_{z_i} \rbrace \mathbbm{1}\lbrace B_{z_i}\rbrace\mathbb{E}\left(  \prod_{i=1}^{q}\mathbbm{1}\lbrace C_{z_i} \cap D_{z_i} \cap F_{z_i}  \rbrace \Bigg\vert \Lambda \right)\right] \\ \nonumber &= \mathbb{E}\left[\prod_{i=1}^{q} \mathbbm{1}\lbrace A_{z_i} \rbrace \mathbbm{1}\lbrace B_{z_i}\rbrace  \prod_{i=1}^{q}\mathbb{E}\left( \mathbbm{1}\lbrace C_{z_i} \cap D_{z_i} \cap F_{z_i}  \rbrace \Bigg\vert \Lambda \right)\right] 
\\ \nonumber &= 
\mathbb{E}\left[\prod_{i=1}^{q} \mathbbm{1}\lbrace A_{z_i} \rbrace  \prod_{i=1}^{q}\mathbb{E}\left( \mathbbm{1}\lbrace B_{z_i} \cap C_{z_i} \cap D_{z_i} \cap F_{z_i}  \rbrace \Bigg\vert \Lambda \right)\right]\\
&=\mathbb{E}\left[\prod_{i=1}^{q} \mathbbm{1}\lbrace R(Q_{6n}(nz_i)) < 6n \rbrace  f(\Lambda_{Q_{6n}(nz_i)})\right], \label{eq11}
\end{align}
where $f(\Lambda_{Q_{6n}(x)}) \coloneqq \mathbb{E}\left(\mathbbm{1}\lbrace B_{x} \cap C_{x} \cap D_{x} \cap F_{x}  \rbrace \,  \vert \, \Lambda \right)$, a bounded measurable deterministic function of $\Lambda_{Q_{6n}(x)}$,
and where
by $\Lambda$-measurability of the events $ \lbrace B_{z_i} \, : \, 1 \leq i \leq q \rbrace$, we  put their indicators into the conditional expectation given $\Lambda$.
Now, the set $\varphi \coloneqq \lbrace nz_1,\ldots,nz_p \rbrace \subset \mathbb{R}^{d}$ is finite and satisfies: $$\forall i \neq j, \lVert nz_i - nz_j \rVert_{\infty} > 18n.$$ Since the infinite norm is always upper bounded by the Euclidean norm, we have $\forall i \neq j, \lVert nz_i - nz_j \rVert_{2} > 18n  $, and so $\varphi$ satisfies:
$$\forall x \in \varphi, \, \text{dist}(x, \varphi \setminus \lbrace x \rbrace) > 18n=3 \times 6n.$$
We can therefore apply condition (3) in Definition~\ref{Def.stabilizing} (with $n$ replaced by $6n$) to get that the random variables appearing in the right-hand side of~\eqref{eq11} are independent. Hence:
\begin{align*}
\label{eq12} \mathbb{E}\left(\prod_{i=1}^{q} \xi_{z_i} \right)  &= \prod_{i=1}^{q} \mathbb{E}\left[\mathbbm{1}\lbrace R(Q_{6n}(nz_i) < 6n \rbrace   f(\Lambda_{Q_{6n}(nz_i)})\right]\\
&= \prod_{i=1}^{q} \mathbb{E}(\xi_{z_i}),
\end{align*}
which concludes the proof of Lemma~\ref{Claim2.supercritical}.
\end{proof}

\begin{lemma}\label{Claim3.supercritical}
For any $r>0$ we have 
\begin{equation*}
    \lim_{n \uparrow \infty}\lim_{ p \uparrow 1,\,\lambda \uparrow \infty}
    \mathbb{P}(\boldsymbol{0} \, \text{is $n$-good}) = 1.
\end{equation*}
\end{lemma}
\begin{proof}
We shall prove that
\begin{equation*}
    \lim_{n \uparrow \infty}\lim_{p \uparrow 1,\;\lambda \uparrow \infty} \mathbb{P}(\boldsymbol{0} \, \text{is $n$-bad}) = 0.
\end{equation*}
Take any $\epsilon>0$.
Denote respectively by $A,B,C,D,F$ the events that the conditions (1), (2), (3), (4), (5) in the definition of $n$-goodness hold for $z=\boldsymbol{0}$.
Denote also by $\tilde A$ the event that 
$R(Q_{6n})< n/2$. Note that $\tilde A\subset A$ and thus we have:
\begin{align*}
    \mathbb{P}(\boldsymbol{0} \, \text{is $n$-bad}) &= \mathbb{P}(A^c \cup B^c \cup C^c \cup D^c \cup F^c)\\
&\le \mathbb{P}(\tilde A^c \cup B^c \cup C^c \cup D^c \cup F^c)
\\ &\le \mathbb{P}(\tilde A^c) + \mathbb{P}(B^c) + \mathbb{P}(B\cap C^c) + \mathbb{P}(D^c)  + \mathbb{P}(\tilde A\cap D \cap F^c).
\end{align*}
First, partitioning the cube $Q_{6n}$ into $12^d$ subcubes $(K_i)_{1 \leq i \leq 12^d}$ of side length $n/2$, we get:
\begin{align}
    \nonumber \mathbb{P}(\tilde A^c) &= 
    \mathbb{P}(R(Q_{6n}) \geq n/2)
            \\ \nonumber &=  
    \mathbb{P}\left(\bigcup_{i=1}^{12^d} \left\lbrace  R(K_i) \geq n/2 \right\rbrace \right)
    \\ \nonumber &\leq 12^d \; \mathbb{P}(R(Q_{n/2}) \geq n/2) \qquad \text{by stationarity of the $R$'s.} 
\end{align}
Therefore, by condition (2) of Definition~\ref{Def.stabilizing}, we get $\lim_{n \uparrow \infty} \mathbb{P}(\tilde A^c) = 0$.
Also,
\begin{align*}
    \mathbb{P}(B^c)=\mathbb{P}(E \cap Q_n = \emptyset)
\end{align*}
and thus $\lim_{n \uparrow \infty} \mathbb{P}( B^c) = 0$.
Fix  $n$ large enough such that $\mathbb{P}(\tilde A^c)\le \epsilon/5$
and $\mathbb{P}( B^c)\le \epsilon/5$. For such $n$, $Q_n$, $Q_{3n}$ and $Q_{6n}$ intersect almost surely zero or a finite number of edges and vertices.

Let's now deal with the quantity $\mathbb{P}(B \cap C^c)$. We have:
\begin{align*}
    \mathbb{P}( B\cap C^c)=
     \mathbb{P}(E \cap Q_n \neq \emptyset\,\text{and}\,
    \forall \, e\in E\cap Q_n\; :    e \, \text{is closed} ).
\end{align*}
This latter probability converges to~0 when $\lambda\to\infty$ (for fixed $n$ and $r>0$). Hence, for large enough $\lambda<\infty$ (depending on $n,r$) we have $\mathbb{P}(B\cap C^c)\le \epsilon/5$.
Similarly, 
\begin{align*}
        \mathbb{P}(D^{c})
    &=\mathbb{P}(\exists v\in V\cap Q_{6n}:\, v \text{ is closed}) 
\end{align*}
converges to~0 when $p\uparrow1$ (for fixed $n$ and $r>0$),
hence, for large enough  $p<1$, we have 
$\mathbb{P}(D^c)\le \epsilon/5$.

Regarding the event $\tilde A\cap D\cap F^c$,
 note that under the event $\tilde A$, we have $$R(Q_{6n}) < n/2 < 3n/2.$$ Hence, by asymptotic essential connectedness (see Definition~\ref{Def.eac}), we have that $\text{supp}(\Lambda_{Q_{3n}}) \neq \emptyset$ and, moreover, there exists a connected component $\Delta$ of $\text{supp}(\Lambda_{Q_{6n}})$ such that $\text{supp}(\Lambda_{Q_{3n}}) \subset \Delta \subset \text{supp}(\Lambda_{Q_{6n}})$. Therefore
 $$\tilde A\cap D\cap F^c\subset 
(\exists e\in E\cap Q_{6n}:\, e \text{ is closed }).$$ 
 Clearly, for fixed $n,r$ and independently of $p$,
 $$\lim_{\lambda\to\infty} \mathbb{P}(\exists e\in E\cap Q_{6n}:\, e \text{ is closed })=0.$$
Hence, we can find $\lambda<\infty$ large enough (depending on $n,r$) such that 
$\mathbb{P}(\tilde A\cap D\cap F^c)\le \epsilon/5$.
Since $\epsilon>0$ was arbitrary,  this concludes the proof of Lemma~\ref{Claim3.supercritical}.
\end{proof}

By Lemmas~\ref{Claim2.supercritical} and~\ref{Claim3.supercritical}, 
using~\cite[Theorem 0.0]{liggett_domination_1997}, the process of $n$-good sites is 
stochastically dominated from below by a supercritical Bernoulli process
for large enough $n<\infty, \lambda<\infty,p<1$. Thus, we can make the process of $n$-good sites percolating.
By Lemma~\ref{Claim1.supercritical},
the connectivity graph $\mathcal{G}$ with  these values of $\lambda,p$ percolates, thus concluding the proof of Theorem~\ref{Thm.analogue:supercritical}.

\subsection{Proof of Proposition~\ref{prop-non-triviality-PVT-site-threshold}}
We first prove that $p^* <1$ as a consequence of Theorem~\ref{Thm.analogue:supercritical}. Then, using an adapted path-count argument, very much as in classical i.i.d. percolation on the square grid (see e.g.~\cite[Theorem 1.1]{meester_continuum_1996}), we show that $p^*>0$. \\

\noindent $\boxed{p^*<1}$ By Theorem~\ref{Thm.analogue:supercritical}, we can find large enough $p \in (0,1)$ such that $P(p,U,H)>0$ for any $H \in \left[0,\infty \right)$ and large enough $U < \infty$ depending on the chosen $H$. Choosing $H=0$, we thus obtain the existence of some $p \in (0,1)$ such that $P(p,U,0)>0$ for some $U < \infty$. Now, as noted in Section~\ref{sss.PVT-site-percolation}, the case $H=0$ corresponds to PVT site percolation and the percolation probability $P(p,U,H)=:P_{PVT}(p)$ does not depend on $U$. Thus, $P_{PVT}(p) >0$ for some $p < 1$, and so $p^*<1$. \\

\noindent $\boxed{p^*>0}$
It is known that the degree of all vertices (i.e. 0-facets) of a $d$-dimensional PVT generated by a homogeneous Poisson point process in $\mathbb{R}^d$ is almost surely equal to $d+1$, as a consequence of the following facts:
\begin{itemize}
    \item Such a PVT is almost surely regular and normal (see~\cite[Proposition PV2]{okabe_spatial_1992}).
    \item The cells of such a PVT are almost surely polytopes, i.e. convex and bounded polyhedra (see~\cite[Proposition PV1]{okabe_spatial_1992}).
    \item Each $s$-facet of a Voronoi cell in $\mathbb{R}^d$ lies on a $s$-dimensional hyperplane whose points are equidistant of $d-s+1$ atoms of the Poisson point process having generated the PVT (this is a consequence of~\cite[Property IV3]{okabe_spatial_1992} and~~\cite[Property IV4]{okabe_spatial_1992}).
\end{itemize}
 
Moreover, $S$, being a regular and normal tessellation, is locally finite. This observation combined with the degree bound allows to use an adapted path-count argument, as follows. \\

First, denote by  $\Tilde{\mathcal{G}_p}=:\mathcal{G}_{p,0,0}$ the random graph arising from the PVT site percolation process with parameter $p$ and note that percolation of $\Tilde{\mathcal{G}_p}$ is equivalent to the percolation of $\mathcal{G}_{p,U,0}$ whenever $p \in \left[0,1\right]$ and $U \geq 0$ (indeed, the fact that $H=0$ makes the percolation independent of the Cox points in our model). Hence:

\begin{equation}
    \label{eq-chap4-PVT-site-percolation} 
       P_{PVT}(p)=\mathbb{P}\left( \Tilde{\mathcal{G}}_p \, \text{ has an infinite connected component}\right).
\end{equation}

Denote by $\Phi$ the point process of crossroads (i.e. vertices of the PVT $S$) and denote by $\mathbb{P}^0$ its Palm probability. Since $Y$ is a doubly stochastic Bernoulli point process supported by the crossroads, the conditional distribution of $Y$ given $S$ is the same under the stationary probability $\mathbb{P}$ and under the Palm probability $\mathbb{P}^0$. As a matter of fact, for every crossroad $v \in V$, we have: $\mathbb{P}
^0(Y(\lbrace v \rbrace) >0 \, \vert \, S) = \mathbb{P}(Y(\lbrace v \rbrace) >0 \, \vert \, S) = p$. Moreover, conditionally to the realisation of the PVT $S$, the states of distinct crossroads (i.e. open or closed) remain independent. \\

Fix some $n \geq 1$. A \emph{self-avoiding path} $\gamma$ of length $n$ starting from the typical crossroad $\boldsymbol{0}$ is a sequence of crossroads $\boldsymbol{0}= v_1, \ldots, v_n \in V$ with $v_i \neq v_j$ for $i \neq j$ and such that $v_i$ and $v_{i+1}$ are adjacent in $S$ whenever $1 \leq i \leq n-1$. If the typical crossroad $\boldsymbol{0}$ belongs to an infinite connected component in $\Tilde{\mathcal{G}_p}$ (which we denote by $\boldsymbol{0} \leadsto \infty$), there must exist such a path with a Bernoulli point present at all crossroads of the path. Denote this event by $A_n$. \\
Then we have:
\begin{equation*}
    \mathbb{P}^0(\boldsymbol{0} \leadsto \infty) \leq \mathbb{P}^0(A_n).
\end{equation*}
Let $SAP_n$ denote the set of self-avoiding paths of length $n$ starting from the typical crossroad $\boldsymbol{0}$. By the union bound, we have:
\begin{align*}
    \mathbb{P}^0(A_n) &\leq \mathbb{E}^0 \left[ \sum_{(v_1, \ldots, v_n) = \gamma \in SAP_n} \mathbb{P}^0 \left(  \bigcap_{i=1}^n \lbrace Y(\lbrace v_i \rbrace) >0 \rbrace \, \Big\vert \, S\right)\right] \\ &= \mathbb{E}^0 \left[ \sum_{(v_1, \ldots, v_n) = \gamma \in SAP_n} p^n \right] \\ &= \mathbb{E}^0 \left[ \#(SAP_n)p^n \right],
\end{align*}
where $\#(SAP_n)$ denotes the cardinal of $SAP_n$ and where we have used the conditional independence of the states of the vertices as well as the conditional distribution of $Y$ given $S$ to get the first equality. Now, using the fact that $\forall v \in V, \deg v = d+1$ a.s., we get that $\#(SAP_n) \leq (d+1) \times d^{n-1}$. Hence:
\begin{equation*}
    \mathbb{P}^0(\boldsymbol{0} \leadsto \infty) \leq (d+1) \times d^{n-1} p^n = \frac{d+1}{d} (dp)^{n}.
\end{equation*}
When $p<1/d$, the quantity in the right-hand side converges to~0 as $n \uparrow \infty$. Hence, for $p<1/d$, we have $\mathbb{P}^0(\boldsymbol{0} \leadsto \infty) = 0$. \\
To conclude that $\Tilde{\mathcal{G}}_p$ does not percolate, we proceed as follows. For a crossroad $v \in V$, denote by $\lbrace v \leadsto \infty \rbrace$ the event that $v$ belongs to an infinite connected component of the PVT site percolation graph $\Tilde{\mathcal{G}}_p$. By Markov's inequality, we have:
\begin{align*}
    \mathbb{P}\left( \Tilde{\mathcal{G}}_p \, \text{ has an infinite connected component}\right) &=: \mathbb{P}(\exists v \in V \, : \, v \leadsto \infty) \\
    &\leq \mathbb{E}\left[ \# \lbrace v \in V \, : \, v \leadsto \infty \rbrace \right],
\end{align*}
and so, by~\eqref{eq-chap4-PVT-site-percolation}, we get:

\begin{equation}
    \label{eq-chap4-getting-back-to-PVT-site-perco-Markov}
    P_{PVT}(p) \leq \mathbb{E}\left[ \# \lbrace v \in V \, : \, v \leadsto \infty \rbrace \right].
\end{equation}
Denote by $\lambda_0$ the intensity of the point process $\Phi$ of crossroads of $S$ and fix some $p<1/d$. By the Campbell-Little-Mecke-Matthes theorem (see \cite[Theorem 6.1.28]{blaszczyszyn2020randommeasures}), we have:

\begin{align*}
    \mathbb{E}\left[ \# \lbrace v \in V \, : \, v \leadsto \infty \rbrace \right] &= \mathbb{E} \left[ \int_{\mathbb{R}^d} \mathbbm{1}\lbrace x \leadsto \infty \rbrace \Phi(dx) \right] \\
    &= \lambda_0 \int_{\mathbb{R}^d} \mathbb{E}^0 \left[ \mathbbm{1} \lbrace \boldsymbol{0} \leadsto \infty \rbrace \right]dx \\
    &= \lambda_0 \int_{\mathbb{R}^d} \mathbb{P}^0(\boldsymbol{0} \leadsto \infty)dx \\
    &=0,
\end{align*}
where we have used the fact that $p < 1/d \Rightarrow \mathbb{P}^0(\boldsymbol{0} \leadsto \infty) = 0$ to get the last equality. By~\eqref{eq-chap4-getting-back-to-PVT-site-perco-Markov}, we have $P_{PVT}(p)=0$ whenever $p<1/d$ and thus $p^* \geq 1/d >0$.

\subsection{Proof of Theorem~\ref{Thm.permanently-super-critical}}
Proving Theorem~\ref{Thm.permanently-super-critical} amounts to showing that $\mathcal{G}$ percolates with positive probability when $\lambda=0$, $p<1$ is sufficiently large  and $r<\infty$ is sufficiently large. 
We thus assume throughout the rest of this subsection that $\lambda=0$, $r$ and $p$ are the varying parameters and we still refer to $\mathcal{G}$ for the associated connectivity graph. 

Say a street $e \in E$ is \emph{hard-geometric-open} if its length is smaller than the connectivity threshold $r$: $\vert e \vert \leq r$. If not, say $e$ is \emph{hard-geometric-closed}.

Once again, we will use a coarse-graining argument. Since the development is very similar to the one exposed in the previous subsection, we only give details on which modifications should be brought to the proof of Theorem~\ref{Thm.analogue:supercritical} to prove Theorem~\ref{Thm.permanently-super-critical}. \\

To this end, we consider now the following percolation model on the integer lattice $\mathbb{Z}^d$. For $n \geq 1$, say a site $z \in \mathbb{Z}^{d}$ is $n$-good if it satisfies the following conditions:
\begin{enumerate}
\item[(1)] $R(Q_{6n}(nz)) < 6n$. 
\item[(2)] $E \cap Q_n(nz) \neq \emptyset$, i.e. the cube $Q_n(nz)$ contains a \emph{full} street.
\item[$(\hat{3})$] $\exists \, e \in E \cap Q_n(nz)$ such that $\vert e \vert \le r$. In other words, there exists a hard-geometric-open street that is fully included in the cube $Q_n(nz)$.
\item[(4)] All crossroads in $Q_{6n}(nz)$ are open, in the sense of Definition~\ref{Def.open/crossroad}.
\item[$(\hat{5})$] Every two hard-geometric-open streets $e,e' \in E \cap Q_{3n}(nz)$ (i.e. such that $\vert e \vert \leq r$ and $\vert e' \vert \leq r$) are connected by a path in $\mathcal{G}\cap Q_{6n}(nz)$.
\end{enumerate}
We say a site $z \in \mathbb{Z}^{d}$ is $n$-bad if it is not $n$-good. \\

Note that this new definition of $n$-goodness is exactly the same as the one given in the proof of Theorem~\ref{Thm.analogue:supercritical} but with conditions (3) and (5) being replaced by $(\hat{3})$ and $(\hat{5})$. In other words, openness is replaced by hard-geometric-openness. \\

Since we are now dealing with hard-geometric openness and all the other conditions are unchanged, the following is straightforward by adapting the proof of Lemma~\ref{Claim1.supercritical}:
\begin{lemma}
\label{Claim1.short.phasetransition.hardgeometric}
Percolation of the process of $n$-good sites implies percolation of the connectivity graph $\mathcal{G}$.
\end{lemma}

In the same way, we get:
\begin{lemma}
\label{Claim2.short.phasetransition.hardgeometric}
For $z \in \mathbb{Z}^{d}$, set $\xi_{z} \coloneqq \mathbbm{1}\lbrace z \, \text{is $n$-good} \rbrace$. Then $(\xi_{z})_{z \in \mathbb{Z}^{d}}$ is an $18$-dependent random field.
\end{lemma}

\begin{proof}
It suffices to adapt the proof of Lemma~\ref{Claim2.supercritical} as follows. 

Denote respectively by $A_z,B_z,\hat{C}_z,D_z,\hat{F}_z$ the events that the conditions (1), (2), $(\hat{3})$, (4) and $(\hat{5})$ in the definition of $n$-goodness hold for $z \in \mathbb{Z}^{d}$. 

Note first that whenever $z \in \mathbb{Z}^{d}$, the indicators $\mathbbm{1} \lbrace A_z \rbrace$, $\mathbbm{1} \lbrace B_z \rbrace$, $\mathbbm{1} \lbrace \hat{C}_z \rbrace$ are all $\Lambda$-measurable. Doing the exact same calculations as in the proof of Lemma~\ref{Claim2.supercritical}, we arrive at dealing with the quantity 

\begin{align}
\nonumber \mathbb{E}\left(\prod_{i=1}^{q} \xi'_{z_i} \right) &= \mathbb{E}\left[\prod_{i=1}^{q} \mathbbm{1}\lbrace A_{z_i} \rbrace \mathbbm{1}\lbrace B_{z_i}\rbrace \mathbbm{1} \lbrace \hat{C}_{z_i} \rbrace \mathbb{E}\left(  \prod_{i=1}^{q}\mathbbm{1}\lbrace D_{z_i} \cap \hat{F}_{z_i}   \rbrace \Bigg\vert \Lambda \right)\right].
\end{align}
Now, note that conditioned on $\Lambda$, for each $1 \leq i \leq q$, the event $ D_{z_i} \cap \hat{F}_{z_i}  $ only depends on the configuration of  $Y$ inside of the cube $Q_{6n}(nz_i)$. We can thus proceed as in the aforementioned proof by using the complete independence of $Y$ (recall that, given $\Lambda$, $Y$ has the distribution of a Bernoulli point process). Finally, it is clear that $\mathbbm{1} \lbrace \hat{C}_z \rbrace$ is a bounded deterministic function of $\Lambda_{Q_n(nz)}$ and that $\mathbb{E}\left( \mathbbm{1}\lbrace  D_{z} \cap \hat{F}_{z}  \rbrace \Bigg\vert \Lambda \right)$ is a bounded deterministic function of $\Lambda_{Q_{6n}(nz)}$. This allows to proceed exactly as in the proof of Lemma~\ref{Claim2.supercritical} and conclude.
\end{proof}
Finally, for the hard-geometric model, we still have:
\begin{lemma}\label{Claim3.short.phasetransition.hardgeometric}
\begin{equation*}
    \lim_{n \uparrow \infty}\lim_{ p\uparrow 1, r \uparrow \infty}
    \mathbb{P}(\boldsymbol{0} \, \text{is $n$-good}) = 1.
\end{equation*}
\end{lemma}
\begin{proof}
Again, we shall prove that:
\begin{equation*}
    \lim_{n \uparrow \infty}\lim_{p \uparrow 1, r \uparrow \infty} \mathbb{P}(\boldsymbol{0} \, \text{is $n$-bad}) = 0.
\end{equation*}
Take any $\epsilon>0$. We adapt the proof of Lemma~\ref{Claim3.supercritical} as follows.
Denote respectively by $A,B,\hat{C},D,\hat{F}$ the events that the conditions (1), (2), $(\hat{3})$, (4) and $(\hat{5})$ in the definition of $n$-goodness hold for $z=\boldsymbol{0}$.
Denote also by $\tilde A$ the event that 
$R(Q_{6n})< n/2$. As in the aforementioned proof, we have
\begin{align*}
    \mathbb{P}(\boldsymbol{0} \, \text{is $n$-bad})  &\le \mathbb{P}(\tilde A^c) + \mathbb{P}(B^c) + \mathbb{P}(B \cap \hat{C}^c) + \mathbb{P}(D^c) + \mathbb{P}(\tilde A \cap D \cap \hat{F}^c). 
\end{align*}
In the above inequality, we deal with the first, second and fourth quantities as before and so  we can fix $n$ large enough such that $\mathbb{P}(\tilde A^c)\le \epsilon/5$
and $\mathbb{P}(B^c)\le \epsilon/5$. For such $n$, $Q_n$, $Q_{3n}$ and $Q_{6n}$ intersect almost surely zero or finitely many edges and vertices. We can then fix $p<1$ large enough such that $\mathbb{P}(D^c)\le \epsilon/5$.

Let's now deal with the quantity $\mathbb{P}(B \cap \hat{C}^c)$. We have:

\begin{align*}
    \mathbb{P}(B \cap \hat{C}^c) &=
     \mathbb{P}(E \cap Q_n \neq \emptyset\;\text{and}\;
    \forall \, e\in E\cap Q_n\;    \vert e \vert > r ) \\ &= \mathbb{E}\left(\mathbbm{1}\lbrace E \cap Q_n \neq \emptyset \rbrace \prod_{e \in E\cap Q_n} \mathbbm{1}\lbrace \vert e \vert > r \rbrace \right).
\end{align*}

Note first that on the event $\lbrace E \cap Q_n \neq \emptyset \rbrace$, the latter product is non-empty. Moreover, since $E\cap Q_n$ contains finitely many edges (recall that $n$ is fixed) and since we have $$\forall e \in E, \; \lim_{r \uparrow \infty} \mathbbm{1}\lbrace \vert e \vert > r \rbrace = 0 \; \text{a.s.},$$ by dominated convergence, we have that the latter expectation  converges to~0 when $r\to\infty$ (for fixed $n$). Therefore, $\lim_{r \uparrow \infty} \mathbb{P}(B \cap \hat{C}^c)=0$ (for fixed $n$). 

Regarding the event $\tilde A\cap D \cap \hat{F}^c$, we proceed as before and use asymptotic essential connectedness to get
 \begin{equation*}
     \tilde A\cap D\cap \hat{F}^c \subset (\exists e\in E\cap Q_{6n}:\, e \text{ is hard-geometric-closed}).
 \end{equation*}
 Clearly, for fixed $n$,
\begin{align*} 
\lim_{r\to\infty} \mathbb{P}(\exists e\in E\cap Q_{6n}:\, e \text{ is hard-geometric-closed})& =\lim_{r\to\infty} \mathbb{P}(\exists e\in E\cap Q_{6n}:\, \vert e \vert >r) \\ &=0.
\end{align*}

Hence, we can find $r<\infty$ large enough (depending on $n$) such that 
$\mathbb{P}(B \cap \hat{C}^c)\le \epsilon/5$ and $\mathbb{P}(\tilde A \cap D \cap \hat{F}^c)\le \epsilon/5$. Since $\epsilon >0$ was arbitrary, this concludes the proof of Lemma~\ref{Claim3.short.phasetransition.hardgeometric}.

\end{proof}

By Lemmas~\ref{Claim2.short.phasetransition.hardgeometric} and~\ref{Claim3.short.phasetransition.hardgeometric}, 
using~\cite[Theorem 0.0]{liggett_domination_1997}, the process of $n$-good sites is 
stochastically dominated from below by a supercritical Bernoulli process
for large enough $n<\infty, p<1, r<\infty$. Thus, we can make the process of $n$-good sites percolating.
By Lemma~\ref{Claim1.short.phasetransition.hardgeometric},
the connectivity graph $\mathcal{G}$ with these values of $p$ and $r$ percolates, thus concluding the proof of  Theorem~\ref{Thm.permanently-super-critical}.

\section{Model extensions and concluding remarks}
\label{S.Conclusion}

In this paper, we have introduced and studied a new model for continuum line-of-sight percolation in a random environment. This mathematical model can be seen as a good candidate for the modelling of telecommunications networks in an urban scenario with regular obstructions: the random support (equivalent to the street system of the city) is modelled by a PVT. Users are dropped on the edges of the PVT according to a Cox point process with linear intensity $\lambda$. Obstructive connectivity conditions require the presence of an additional Bernoulli point process (representing relays which could either be real users or physical antennas) at the vertices of the PVT (crossroads of the city) to ensure connectivity between adjacent streets. We have proven via coarse-graining arguments that a minimal relay proportion $p > p^*$ is necessary to allow for percolation of the connectivity graph, that a non-trivial subcritical phase exists whenever the connectivity threshold $r$ is not too large and that a supercritical phase exists for all $r>0$. 
Moreover, we also performed Monte-Carlo simulations to get numerical estimations of critical parameters of our model and estimate the frontiers between the different connectivity regimes and particular cases of our model (PVT site percolation, PVT hard-geometric bond percolation and PVT soft-geometric bond percolation). \\

Our line-of-sight percolation model in a random environment has  many possible generalizations.  
The most obvious one is  the \emph{dual tessellation} of the PVT: a~\emph{Poisson-Delaunay} tessellation~\cite[Section~9.2]{chiu_stochastic_2013}, which is known to  be  stabilizing and essentially asymptotically connected
(see~\cite[Exercise 3.2.7]{jahnel2020probabilistic} and~\cite[Section 3.1]{hirsch2018continuum}).
More interestingly, one can try to prove these stabilization and asymptotic essential connectedness properties (which are fundamental for our approach) for the \emph{generalized Poisson-Voronoi weighted tessellation}~\cite{flimmel2019limit}, including, as special cases, \emph{Laguerre} and \emph{Johnson-Mehl} tessellations. Note that in this latter paper, a different stabilization property is used to prove expectation and variance asymptotics, as well as central limit theorems for unbiased and asymptotically consistent estimators of geometric statistics of the typical cell.
Finally, one may ask whether the Poisson process, which is underlying in all the above models,  can be replaced by a more general point process, sufficiently mixing~\cite{poinas2019mixing} or having a sufficiently fast decay of correlations~\cite{blaszczyszyn2019limit}.
These latter properties (mixing property and fast decay of correlations) are originally used or developed for the aforementioned context of the limit theory for the typical cell and it is not clear whether they can be used in our percolation context.
Concluding, we believe our model paves the way to the study of a new class of random connection models in a random environment that not only conditions the locations of points but also 
 the connection function. 
As a natural generalization of the LOS connection function,  one can consider connectivity conditions based on two connection radii: one for the nodes of the network being located on the same edge of the random support and another one (typically smaller) for non-line-of-sight (NLOS) connections; i.e. nodes of the network being located on different edges of the random support. This could be of  particular interest for more realistic models of telecommunications networks.

\section*{Acknowledgements}
This research work has been funded by a CIFRE contract between Orange S.A. and Inria. The first author would like to thank Christian Hirsch for many interesting and valuable discussions on the concepts of stabilization and asymptotic essential connectedness. The authors thank the two anonymous referees, whose remarks allowed to substantially improve the quality of the manuscript. 

\bibliographystyle{APT}

\end{document}